\documentclass{proc-l}
	
\usepackage{BoxOperads_Preamble} 
\usepackage{import} 
\usepackage{standalone}
\usepackage[normalem]{ulem}

\title{A Deligne Conjecture for Prestacks}

\author{Ricardo Campos} 
\address[Ricardo Campos]{Institut de Mathématiques de Toulouse, UMR5219, Université de Toulouse, CNRS, UPS, F-31062 Toulouse Cedex 9, France}
\email{ricardo.campos@math.univ-toulouse.fr}  

\author{Lander Hermans}
\address[Lander Hermans]{Universiteit Antwerpen, Departement Wiskunde, Middelheimcampus,
	Middelheimlaan 1,
	2020 Antwerp, Belgium}
	\email{lander.hermans@uantwerpen.be}

\thanks{The first author is supported by the  ANR-20-CE40-0016 HighAGT. The second author is a predoctoral fellow of the Research Foundation - Flanders (FWO),
file number 1194422N. \\
Keywords: operads, prestacks, Deligne conjecture, operadic twisting, Gerstenhaber--Schack complex}

\begin{document}

\subjclass[2020]{Primary 18F20, 18M70, 13D03, 16E40, 55P10}

\date{\today}

\dedicatory{}

\commby{Julia Bergner}

\begin{abstract}
We prove an analog of the Deligne conjecture for prestacks. We show that given a prestack $\mathbb A$, its Gerstenhaber--Schack complex $\mathbf{C}_{\mathsf{GS}}(\mathbb A)$ is naturally an $\E2$-algebra. This structure generalises both the known $\mathsf{L}_\infty$-algebra structure on $\mathbf{C}_{\mathsf{GS}}(\mathbb A)$, as well as the Gerstenhaber algebra structure on its cohomology $\mathbf{H}_{\mathsf{GS}}(\mathbb A)$. 
The main ingredient is the proof of a conjecture of Hawkins \cite{hawkins}, stating that the dg operad $\Quilt$ has vanishing homology in positive degrees. As a corollary, $\Quilt$ is quasi-isomorphic to the operad $\Brace$ encoding brace algebras. In addition, we improve the $\Linf$-structure on $\Quilt$ by showing that it originates from a $\Prelieinf$-structure lifting the $\Prelie$-structure on $\Brace$ in homology.
\end{abstract}

\maketitle
\setcounter{tocdepth}{1}
\tableofcontents

\section{Introduction}\label{parintro}

In his famous 1993 letter, Deligne conjectured that the Gerstenhaber-structure of Hochschild cohomology  for associative algebras lifts to an $\E2$-structure on the Hochschild complex, that is, the complex is an algebra over a dg operad homotopy equivalent to the chain little disks operad $\Disk$ \cite{deligneletter, gerstenhaber1963, cohenladamay1976}. The many solutions proposed \cite{mccluresmith2002, kontsevichsoibelmandeligne, tamarkin1998, bergerfresse2004, bataninberger2009, voronov, kaufmann2007} factor through Gerstenhaber and Voronov's explicit Homotopy G-structure on the complex \cite{gerstenhabervoronov}, that is, they construct a dg operad $\mathcal{G}$ homotopy equivalent to $\Disk$ and a quasi-isomorphism $\mathcal{G} \stackrel{\sim}\longrightarrow \HG$ where $\HG$ is the dg operad encoding Homotopy G-algebras. \\

In this paper we are interested in an analog of the Deligne conjecture for prestacks. In this setting, {similar to the Hochschild complex for associative algebras,} the Gerstenhaber--Schack complex $\mathbf{C}_{\mathsf{GS}}(\mathbb A)$ for a prestack $\mathbb A$ controls the deformations of $\mathbb A$ {and its homology carries a Gerstenhaber algebra structure \cite{gerstenhaberschack1, lowenvandenberghCCT, DVL}}. Our main result is an explicit solution of the Deligne conjecture for prestacks lifting the Gerstenhaber algebra structure to the level of the complex.  
\begin{constr}[Theorem \ref{thmmodel}]
{ There is a dg operad $\TwQuilt$ that is quasi-isomorphic to both $\HG$ and $\Disk$ and which admits an explicit combinatorial description.}
\end{constr} 
\begin{theorem}[Theorem \ref{thmaction}]
{Given a prestack $\mathbb A$, there is an action of $\TwQuilt$ on its Gerstenhaber--Schack complex $\mathbf{C}_{\mathsf{GS}}(\mathbb A)$ inducing a Gerstenhaber algebra structure on Gerstenhaber--Schack cohomology $\mathbf{H}_{\mathsf{GS}}(\mathbb A)$.}
\end{theorem}
We point out that the abstract existence of solutions of the Deligne conjecture for prestacks can be deduced via homotopy transfer (see Remark 4.3). On the other hand, our solution (the first \emph{explicit} one) has both
practical and conceptual advantages. Conceptually, $\TwQuilt$ is a $2$-dimensional analog of $\HG$ which has proven to take up a central role in the classical Deligne conjecture (see above). 

\subsection{The GS complex}

The GS complex for prestacks takes up the pivotal role of the Hochschild complex for associative algebras: its cohomology is a derived invariant computing $\Ext$-cohomology \cite{DVL, lowenvandenberghCCT} and it is endowed with an $\Linf$-structure governing its deformations \cite{vanhermanslowen2022, dinhvanhermanslowen2023}. Prestacks generalize presheaves of associative algebras by relaxing their functoriality up to a natural isomorphism $c$ called \emph{twists}. They are motivated by (noncommutative) algebraic geometry, where they appear for example as structure sheaves of a scheme and noncommutative deformations thereof \cite{artintatevandenbergh, barannikov2007, dinhvanliulowen2017, lowenvandenberghhoch, vandenbergh2}. Indeed, Lowen and Van den Bergh observed in \cite{lowenvandenberghhoch} that $\Ext$-cohomology of presheaves parametrizes their first order deformations, not as presheaves, but as prestacks. In a more global picture, they have become part of homological mirror symmetry as proposed by Kontsevich \cite{kontsevich2, aurouxorlov2008}.

\subsection{Structure of the proof}
Let us start by recalling Gerstenhaber and Voronov's approach and present our key insight. For an associative algebra $A$, the Homotopy G-structure on its (desuspended) Hochschild complex $\Choch(A)$ is obtained by twisting the brace-structure with the multiplication. Using operadic twisting \cite{dolgushevwillwacher2015, dotsenkoshadrinvallette2024}, this result can be rephrased as a morphism of dg operads 
$$\TwBrace \longrightarrow \End(s^{-1}\Choch(A))$$
where $\TwBrace$ {(Definition \ref{deftwbrace})} is the operadic twisting of the operad $\Brace$ encoding brace-algebras. In fact, $\HG$ is isomorphic to  $\TwBraceover$ {(Remark \ref{rem:historical remark})}, a quasi-isomorphic suboperad of $\TwBrace$. The same approach does not work for prestacks: the operad $\Brace$ is too small to act on the GS complex $\CGS(\A)$ of a prestack $\A$. Indeed, a brace-algebra induces a $\Lie$-structure although a $\Linf$-structure is required to capture prestack structures as Maurer--Cartan elements. 

As a remedy, Dinh Van, Lowen and the second author construct in \cite{vanhermanslowen2022} an action of Hawkins' dg operad $\Quilt$ \cite{hawkins} on the (desuspended) GS complex $\CGS(\A)$. As $\Quilt$ projects onto $\Brace$ and induces a $\Linf$-structure, it is posited as a suitable replacement. In \S \ref{parquilt}, our main technical result shows the following, hereby proving Hawkins' conjecture \cite[Conj. 3.7]{hawkins}.
\begin{theorem}[Theorem \ref{thmhomologyQuilt}]
The projection $\Quilt \twoheadrightarrow \Brace$ is a quasi-isomorphism.
\end{theorem}
{ In addition, we improve upon the result from \cite[Thm. 7.8]{hawkins} which constructs a morphism $\Linf \longrightarrow \Quilt$. Indeed, we factor this morphism through $\Prelieinf$, the minimal model of the Koszul operad $\PreLie$, lifting the morphism $\PreLie \longrightarrow \Brace$ in homology.}
\begin{prop}[Proposition \ref{propprelieinf}] 
{ We have a morphism $\Prelieinf\longrightarrow \Quilt$ inducing the morphism $\PreLie \longrightarrow \Brace$ in homology.}
\end{prop}

Our {next} key insight is that we can now apply the machinery of operadic twisting in \S \ref{parmodel}. As the twisting functor $\Tw$ preserves quasi-isomorphisms \cite[Thm. 5.1]{dolgushevwillwacher2015}, we obtain our main result.
\begin{theorem}[Corollary \ref{cormodel}]
$\TwQuilt$ is an $\E2$-operad.
\end{theorem}

Observe that generally the $\Prelieinf$-structure does not carry through to $\TwQuilt$-algebras. An appropriate analogy is the fact that the bracket in the Hochschild complex arises from a $\PreLie$-algebra, which, after twisting by the relevant Maurer--Cartan element, is no longer a dg $\PreLie$-algebra, only a dg $\Lie$-algebra. In the present setting, after twisting, the relevant algebraic structure is $\Linf$, instead of $\Prelieinf$.

Finally, in \S \ref{paraction}, we show that the action of $\Quilt$ on the GS complex from \cite{vanhermanslowen2022} extends to an action of $\TwQuilt$ by twisting with the prestack's twists $c$. Hence, we obtain the following explicit solution to the Deligne conjecture for prestacks.
\begin{theorem}[Theorem \ref{thmaction}]
There is an action of the $\E2$-operad $\TwQuilt$ on the GS complex $\CGS(\A)$ of a prestack $\A$.
\end{theorem} 
Remark that in \cite{vanhermanslowen2022} they `informally' twist $\Quilt$ by $c$ and establish an action of a new operad $\Quilt_b[\![c]\!]$ to obtain the correct $\Linf$-structure. This is subsumed in our $\TwQuilt$-action as we construct a morphism $\TwQuilt \twoheadrightarrow \Quilt_b[\![c]\!]$ through which it factors.

Interestingly, in \S \ref{parparactionpresheaves}, we obtain as a bonus, in the restricted case of presheaves, a second `orthogonal' $\TwQuilt$-action and thus solution to the Deligne conjecture. In particular, this action subsumes the action of Hawkins' operad $\mQuilt$ for presheaves from which he deduces the correct $\Linf$-structure \cite{hawkins}.

\medskip

\noindent \emph{Conventions.}

We work over a field of characteristic zero even though the results of Section \ref{parquilt} and in particular Hawkins' conjecture hold over the integers with no modifications to the proofs. 
We use cohomological conventions throughout. In particular, chains on a topological space live in non-positive degrees and have a differential of degree $+1$.

If $\sigma$ is a permutation, we use $(-1)^\sigma$ to denote its sign and $(-1)^{k+\sigma}$ should be interpreted as $(-1)^k(-1)^\sigma$.

\section{The operad $\Quilt$ and its homology}\label{parquilt}

\subsection{The operad $\Quilt$}\label{parparquilt}
In this section we recapitulate the dg operad $\Quilt$ introduced by Hawkins \cite{hawkins} and fix conventions.

\subsubsection{The operad $\Brace$}\label{parparparbrace}

The operad $\Brace$ encoding brace algebras is defined using trees, that is, planar rooted trees. Following the presentation from \cite[\S 2.2]{hawkins} and \cite[\S 2.1]{vanhermanslowen2022}, a tree $T=(V_T,E_T,\tre_T)$ consists of a set of vertices $V_T$, a set of edges $E_T$ which induces a ``vertical'' partial order $<_T$ on $V_T$, and a ``horizontal'' partial order $\tre_T$ on $V_T$, satisfying a number of properties \cite[Def. 2.3]{hawkins}. For $(u,v) \in E_T$, we call $u$ the \emph{parent} of $v$ and $v$ a \emph{child} of $u$. A \emph{leaf} is a vertex without a child. We have an induced total order on $V_T$ by setting $u\nearrow_T v$ if $u<_T v$ or $u\tre_T v$. We will depict the vertical and horizontal orders in the plane as follows
$$ \text{below} <_T \text{above} \quad \text{ and } \quad \text{left} \tre_T \text{right}.  $$
with the root at the bottom. This corresponds to the convention of \cite{vanhermanslowen2022} and reverses the direction of $<_T$ in \cite{hawkins}.

For $n\geq 1$, let $\Tree(n)$ denote the set of {planar rooted} trees with vertex set $\{1,\ldots,n\}$ labelled vertices. For example, $\Tree(3)$ contains a total of $12$ elements corresponding to the different labelings of
$$\scalebox{0.7}{$
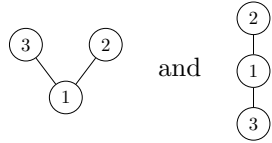
\begin{tikzpicture}
	\begin{pgfonlayer}{nodelayer}
		\node [style=circle vertex] (0) at (0, -1) {$1$};
		\node [style=circle vertex] (1) at (-1.5, 1) {$3$};
		\node [style=circle vertex] (2) at (1.5, 1) {$2$};
	\end{pgfonlayer}
	\begin{pgfonlayer}{edgelayer}
		\draw (0) to (1);
		\draw (0) to (2);
	\end{pgfonlayer}
\end{tikzpicture}
$} \quad \text{ and } \quad \scalebox{0.7}{$\begin{tikzpicture}
	\begin{pgfonlayer}{nodelayer}
		\node [style=circle vertex] (0) at (0, 0) {$1$};
		\node [style=circle vertex] (1) at (0, -2) {$3$};
		\node [style=circle vertex] (2) at (0, 2) {$2$};
	\end{pgfonlayer}
	\begin{pgfonlayer}{edgelayer}
		\draw (0) to (1);
		\draw (0) to (2);
	\end{pgfonlayer}
\end{tikzpicture}
$} .$$

Let $\Brace(n)$ be the free $k$-module on $\Tree(n)$ endowed with the symmetric action by permuting its vertices and the operadic composition is given by substitution of trees into vertices, as follows. For trees $T \in \Tree(m)$, $T' \in \Tree(n)$ and $1 \leq i \leq m$, we denote by $\Ext(T,T',i) \subseteq \Tree(m + n  -1)$ the set of trees extending $T$ by $T'$ at $i$ (that is, $U \in \Ext(T,T',i)$ has $T'$ as a subtree which upon removal reduces to the vertex $i$ of $T$). We then define
$$T \circ_{i} T' := \sum_{U \in \Ext(T,T',i)} U.$$
Consider the following example
$$ \scalebox{0.7}{$
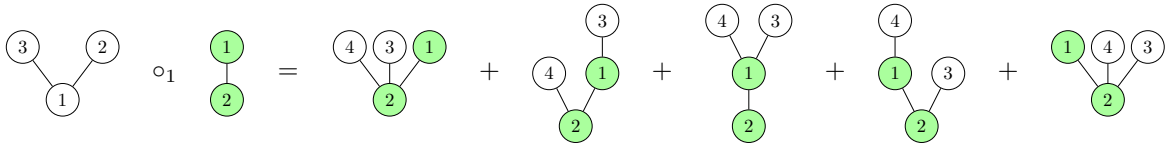
\begin{tikzpicture}
	\begin{pgfonlayer}{nodelayer}
		\node [style=circle vertex] (0) at (0, -1) {$1$};
		\node [style=circle vertex] (1) at (-1.5, 1) {$3$};
		\node [style=circle vertex] (2) at (1.5, 1) {$2$};
	\end{pgfonlayer}
	\begin{pgfonlayer}{edgelayer}
		\draw (0) to (1);
		\draw (0) to (2);
	\end{pgfonlayer}
\end{tikzpicture}
$} \quad \circ_1 \quad \scalebox{0.7}{$\begin{tikzpicture}
	\begin{pgfonlayer}{nodelayer}
		\node [style=green dot] (0) at (0, -1) {$2$};
		\node [style=green dot] (2) at (0, 1) {$1$};
	\end{pgfonlayer}
	\begin{pgfonlayer}{edgelayer}
		\draw (0) to (2);
	\end{pgfonlayer}
\end{tikzpicture}
$} \quad = \quad \scalebox{0.7}{$\begin{tikzpicture}
	\begin{pgfonlayer}{nodelayer}
		\node [style=green dot] (0) at (0, -1) {$2$};
		\node [style=circle vertex] (1) at (0, 1) {$3$};
		\node [style=circle vertex] (2) at (-1.5, 1) {$4$};
		\node [style=green dot] (3) at (1.5, 1) {$1$};
	\end{pgfonlayer}
	\begin{pgfonlayer}{edgelayer}
		\draw (0) to (1);
		\draw (3) to (0);
		\draw (0) to (2);
	\end{pgfonlayer}
\end{tikzpicture}
$} \quad + \quad \scalebox{0.7}{$\begin{tikzpicture}
	\begin{pgfonlayer}{nodelayer}
		\node [style=green dot] (0) at (0, -2) {$2$};
		\node [style=green dot] (1) at (1, 0) {$1$};
		\node [style=circle vertex] (2) at (1, 2) {$3$};
		\node [style=circle vertex] (3) at (-1, 0) {$4$};
	\end{pgfonlayer}
	\begin{pgfonlayer}{edgelayer}
		\draw (0) to (1);
		\draw (3) to (0);
		\draw (2) to (1);
	\end{pgfonlayer}
\end{tikzpicture}
$} \quad + \quad \scalebox{0.7}{$\begin{tikzpicture}
	\begin{pgfonlayer}{nodelayer}
		\node [style=green dot] (0) at (0, -1) {$2$};
		\node [style=green dot] (1) at (0, 1) {$1$};
		\node [style=circle vertex] (2) at (-1.5, 1) {$4$};
		\node [style=circle vertex] (3) at (1.5, 1) {$3$};
	\end{pgfonlayer}
	\begin{pgfonlayer}{edgelayer}
		\draw (0) to (1);
		\draw (3) to (0);
		\draw (0) to (2);
	\end{pgfonlayer}
\end{tikzpicture}
$} \quad + \quad \scalebox{0.7}{$\begin{tikzpicture}
	\begin{pgfonlayer}{nodelayer}
		\node [style=green dot] (0) at (0, -2) {$2$};
		\node [style=green dot] (1) at (0, 0) {$1$};
		\node [style=circle vertex] (2) at (-1, 2) {$4$};
		\node [style=circle vertex] (3) at (1, 2) {$3$};
	\end{pgfonlayer}
	\begin{pgfonlayer}{edgelayer}
		\draw (0) to (1);
		\draw (2) to (1);
		\draw (3) to (1);
	\end{pgfonlayer}
\end{tikzpicture}
$} \quad + \quad \scalebox{0.7}{$\begin{tikzpicture}
	\begin{pgfonlayer}{nodelayer}
		\node [style=green dot] (0) at (0, -2) {$2$};
		\node [style=green dot] (1) at (-1, 0) {$1$};
		\node [style=circle vertex] (2) at (-1, 2) {$4$};
		\node [style=circle vertex] (3) at (1, 0) {$3$};
	\end{pgfonlayer}
	\begin{pgfonlayer}{edgelayer}
		\draw (0) to (1);
		\draw (3) to (0);
		\draw (2) to (1);
	\end{pgfonlayer}
\end{tikzpicture}
$} \quad + \quad \scalebox{0.7}{$\begin{tikzpicture}
	\begin{pgfonlayer}{nodelayer}
		\node [style=green dot] (0) at (0, -1) {$2$};
		\node [style=circle vertex] (1) at (0, 1) {$4$};
		\node [style=green dot] (2) at (-1.5, 1) {$1$};
		\node [style=circle vertex] (3) at (1.5, 1) {$3$};
	\end{pgfonlayer}
	\begin{pgfonlayer}{edgelayer}
		\draw (0) to (1);
		\draw (3) to (0);
		\draw (0) to (2);
	\end{pgfonlayer}
\end{tikzpicture}
$}   $$
In particular, the tree on two vertices $C_2:= \scalebox{0.7}{$\begin{tikzpicture}
	\begin{pgfonlayer}{nodelayer}
		\node [style=circle vertex] (0) at (0, -1) {$1$};
		\node [style=circle vertex] (2) at (0, 1) {$2$};
	\end{pgfonlayer}
	\begin{pgfonlayer}{edgelayer}
		\draw (0) to (2);
	\end{pgfonlayer}
\end{tikzpicture}
$}$ induces a $\Lie$-structure, i.e. we have a morphism
$$\Lie \longrightarrow \Brace,\quad l_2 \longmapsto C_2 - C_2^{(12)}.$$ 
The induced $\Lie$ bracket on a brace algebra is more commonly known as the Gerstenhaber bracket.

\subsubsection{The operad $\FSurj$}\label{parparparFS}

{In this section we recall the operad $\FSurj$, which is a particular model of an $\E2$-operad. 
	 Recall that in characteristic $0$ the operad $\mathsf{E}_2$ is formal \cite{lambrechts2014formality} and therefore $\FSurj$ actually encodes Gerstenhaber algebras up to homotopy.
	As the notation suggests, $\FSurj$ is the second filtration of the surjection operad $\mathsf{Surj}$, which is an $\mathsf{E}_\infty$-operad introduced in \cite{bergerfresse2004} even though we will not work with $\mathsf{Surj}$ in the present paper.}

Again, we largely follow the exposition from \cite[\S 2.3]{hawkins} and \cite[\S 2.2]{vanhermanslowen2022}, reversing the degree in order to work cohomologically. 

Given a set $A$, a \emph{word} over $A$ is an element of the free monoid on $A$. For a word $W = a_1a_2 \dots a_k$, denoting $\langle k \rangle=\{1,\dots,k\}$ we can associate to it the function $W: \langle k \rangle \longrightarrow A: i \longmapsto a_i$, the $i$-th \emph{letter} of $W$ is the couple $(i, a_i)$. We will often identify a word with its graph $W = \{ (i, a_i) \,\, |\,\, i \in \langle n \rangle\} \subseteq \langle n \rangle \times A$, writing $(i, a_i) \in W$.

For $a \in A$, a letter $(i,a) \in W$ is called an \emph{occurrence} of $a$ in $W$. The letter $(i,a)$ is a \emph{caesura} if there is a later occurrence of $a$ in $W$, that is, a letter $(j,a)$ with $i < j$. We say that $a \in A$ is \emph{interposed} in $W$ by $b$ if $W = \dots ba \dots b \dots$. \emph{length} of $W: \langle k \rangle \longrightarrow A$ is $|W| = n$.\\

Let $\FWord(n)$ be the set of words over $\langle n \rangle$ such that:
\begin{enumerate}
\item $W: \langle k \rangle \longrightarrow \langle n \rangle$ is surjective,
\item $W \neq \ldots {u}{u} \ldots$ (nondegeneracy), and
\item for any $u\neq v \in \langle n \rangle$, $W \neq \ldots {u} \ldots {v} \ldots {u} \ldots {v} \ldots$ (no interlacing).
\end{enumerate}
A word $W \in \FWord(n)$ induces two partial orders on $\langle n \rangle$: set $u<_W v$ if $W = \ldots u \ldots v \ldots u \ldots$, and $u \tre_W v$ if all occurrences of $u$ are left of the occurrences of $v$. We call $u$ \emph{a parent of} $v$ and $v$ \emph{a child of} $u$ if $u<_W v$ and they are minimal for this relation: there exists no number $w$ such that $u <_W w <_W v$ holds. We call $u$ a \emph{leaf} if it has no children, that is, it is maximal for $<_W$. Moreover, $u \rightarrow_W v$ if $u<_W v$ or $u\tre_W v$ is a total order.

Let $\FSurj(n)$ be the free $k$-module on $\FWord(n)$ endowed with the symmetric $\Ss_n$-action by permuting letters, i.e. $W^{\si} = \si\inv W$. It is naturally graded by setting $\deg(W) := n - |W|$. \\

The operadic composition on $\FSurj$ is based upon merging of words, as follows. For words $W\in \FWord(m), W' \in \FWord(n)$ and $1 \leq i \leq m$, we denote by $\Ext(W,W',i) \subseteq \FWord(m+n-1)$ the set of extensions of $W$ by $W'$ at $i$ (that is, $X \in \Ext(W,W',i)$ if up to relabelling and deleting repetitions, $W'$ is a subword of $X$ and upon collapsing the letters from $W$ to $i$, relabelling and deleting repetitions, we recover $W$). 

In order to define the composition, we need the sign of an extension. 
\medskip

\noindent \emph{Sign of Extension.}
Let $W\in \FSurj(m)$ and let $\int(W)$ be the set of elements of $\langle m \rangle$ interposed in $W$ ordered by their first occurrence in $W$. For $X \in \Ext(W,W',i)$ the relabelling gives rise to two functions
$\lh m' \rh \overset{\al}{\hookrightarrow} \lh m+m'-1 \rh \overset{\be}{\twoheadrightarrow} \lh m \rh$ which induce functions $\al:\int(W')\longrightarrow \int(X)$ and $\ga: \int(W) \longrightarrow \int(X)$ where $\ga := \be\inv$ except if $i$ is interposed in $W$, then $\gamma(i) := \al(a)$ for $(1,a)$ the first letter of $W'$.
 
As $|\int(W)| = \deg(W)$, an extension $X$ defines a unique $(\deg(W),\deg(W'))-$shuffle $\chi$ and we define 
$$\sgn_{W,W',i}(X) := (-1)^{\chi}$$

Moreover, there is a natural notion of a boundary of a word, which induces a differential.

\medskip

\noindent \emph{Boundary.}
Given a word $W \in \FWord(n)$ and a letter $(i,a)$ of $W$ for which $a$ is repeated in $W$, then define $\partial_{i}W\in \FSurj(n)$ as the word obtained by deleting the letter $(i,a)$ from $W$ (and relabelling). If $a$ is not repeated, then set $\partial_{i}W= 0$.

\medskip

\noindent \emph{Sign of Deletion.}
Given a word $W \in \FWord(n)$ of length $n$, we define $\sgn_{W}:\langle n \rangle \longrightarrow \{-1,1\}$ by setting $\sgn_{W}(i)= (-1)^{k}$ if $(i,a_i)$ is the $k$-th caesura of $W$, and otherwise $\sgn_{W}(i)= (-1)^{k+1}$ if it is the last occurrence, but the previous occurrence is the $k$-th caesura of $W$.

\medskip

The $\Ss$-module $\FSurj$ defines a dg operad with operadic composition given by 
$$ W \circ_{i} W' := \sum_{X \in \Ext(W,W',i)} \sgn_{W,W',i}(X) X$$
and boundary given by
$$\partial W := \sum_{i \in \langle |W| \rangle} \sgn_{W}(i)\partial_{i}W.$$

\begin{vb}
For words $1232,1213 \in \FWord(3)$, we have
$$ 1232 \circ_2 1213 = 1252324 - 1235324 - 1232524 - 1232454 $$ and $$ \partial(1232) = - 132 + 123 \text{ and } \partial(1213) = -213 + 123 $$
\end{vb}

\begin{lemma}
We have a morphism of operads $\FSurj \longrightarrow \mathsf{Com}$ sending a word $W\in \FWord(n)$ to the point if $|W| =n$, and $0$ otherwise.
\end{lemma}

\subsubsection{The operad $\Quilt$} \label{parparparquilt}

In \cite{hawkins}, Hawkins defines a dg suboperad $\Quilt \subseteq \FSurj \otimes_H \Brace$ which we can rephrase as follows: $\Quilt(n)$ is the $k$-module spanned by $(W,T) \in \FWord(n) \times \Tree(n)$ such that
\begin{enumerate}[label=\roman*.]
	\item (Horizontality) If $u <_T v$, then $u\tre_W v$,
	\item (Verticality) if $u <_W v$, then $v \tre_T u$. 
\end{enumerate}
	In this case, we say $W$ \emph{quilts} $T$. It is clear that $\Quilt(n)$ is closed under the $\mathbb S_n$ and the differential. To see that $\Quilt$ is closed under the operadic composition, see \cite[Lemma 3.3]{hawkins}. For a quilt $Q= (W,T)$, the children of a rectangle $u$ with respect to $W$ are called its \emph{vertical} children (see \S \ref{parparparFS}), and its children with respect to $W$ its \emph{horizontal} children (see \S \ref{parparparbrace}). We denote their union as the \emph{children of $u$}. {The \emph{vertical leaves} of $Q$ are the rectangles $u$ without vertical children. The \emph{horizontal leaves} of $Q$ are the rectangles $u$ without horizontal children, that is, the leaves of the tree $T$.} { A quilt $Q=(W,T)$ is \emph{in standard order} if the total order $\nearrow_T $ on vertices agrees with the natural order on $\{1,\ldots,n\}$.}

$\Quilt$ derives its name from the pictorial presentation of its elements $(W,T)$ as quilts, as follows. Let each vertex correspond to a rectangle in the plane, then a quilt $(W,T) \in \Quilt(n)$ is a planar ordering of $n$ rectangles with possibly shaded regions inbetween and possibly certain horizontal lines are drawn double. The tree $T$ determines the horizontal adjacencies, whereas the word $W$ fixes the vertical adjacencies. Their partial orders on vertices impose on the rectangles the following planarity
\begin{alignat*}{2}
&\text{below} <_W \text{above} \quad &&\text{above} \tre_T \text{below}
 \\
  &\text{left} \tre_W \text{right}  \quad &&\text{left} <_T \text{right}
\end{alignat*}
Remark that in this sense the tree $T$ is drawn by turning 90 degrees clockwise. Each rectangle has at most one rectangle adjacent to its left and below. They can have multiple adjacent rectangles to their right and above. 

The following algorithm describes how to draw a quilt from a quilt $Q=(W,T)$. Draw the vertices of $T$ as rectangles of the following size
\begin{align*}
\text{height rectangle }i = \max\{1 \;, \; \#\text{ {horizontal} leaves to the right of }i\text{ in }T \} \\
\text{width rectangle }i = \max\{ 1 \; , \; \#\text{{vertical} leaves above }i\text{ in }W \}
\end{align*}
Draw the tree $T$ in the plane turning it $90$ degrees clockwise from the drawings in \S \ref{parparparbrace}, its root is now the leftmost rectangle. We order the rectangles vertically into columns inductively:
\begin{enumerate}
\item\label{drawingstep1} For $u_1 \tre_W \ldots \tre_W u_k$ the $<_W$-minimal rectangles, draw $k$ vertical columns and draw a shaded rectangle underneath $u_i$ of the following height
$$ \# \{ w \in \RB(u_i) : \nexists w' \in \RB(u): w <_T w' \} $$
where $\RB(u_i) := \{ w : u_i \tre_T w, w \tre_W u_i \}$ the set of rectangles to the right of and below $u_i$.
\item For $u$ drawn, repeat \eqref{drawingstep1} for $u_1 \tre_W \ldots \tre_W u_{k_u}$ the $<_W$-minimal rectangles above $u$, i.e. the children of $u$ in $W$.
\end{enumerate}
When you get to the leaves, shade the appropriate region above to make the full quilt into a rectangle. Finally, if $W=\ldots uv \ldots wu \ldots$ with no $u$ in between $v$ and $w$, then draw a double horizontal line along the edge of $u$ from the depth of $v$ till the depth of $w$. 

The above algorithm is best understood via examples.
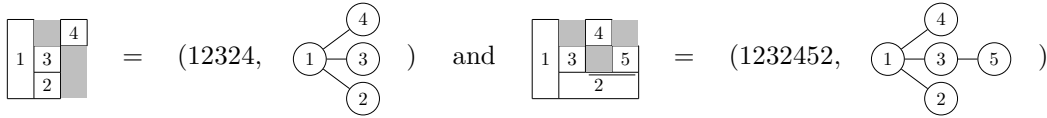
\begin{vbn}\label{exquilt}
$$ \scalebox{0.7}{$\begin{tikzpicture}
	\begin{pgfonlayer}{nodelayer}
		\node [style=none] (0) at (-1.5, 1.5) {};
		\node [style=none] (1) at (-1.5, -1.5) {};
		\node [style=none] (2) at (-0.5, -1.5) {};
		\node [style=none] (3) at (-0.5, 1.5) {};
		\node [style=none] (4) at (-0.5, 0.5) {};
		\node [style=none] (5) at (0.5, 0.5) {};
		\node [style=none] (6) at (0.5, 1.5) {};
		\node [style=none] (7) at (1.5, 1.5) {};
		\node [style=none] (8) at (1.5, 0.5) {};
		\node [style=none] (9) at (0.5, -1.5) {};
		\node [style=none] (10) at (0.5, -0.5) {};
		\node [style=none] (11) at (-0.5, -0.5) {};
		\node [style=none] (12) at (-0.5, -0.5) {};
		\node [style=none] (13) at (-0.475, 0.525) {};
		\node [style=none] (14) at (0.475, 0.525) {};
		\node [style=none] (15) at (0.475, 1.475) {};
		\node [style=none] (16) at (-0.475, 1.475) {};
		\node [style=none] (17) at (0.525, -1.475) {};
		\node [style=none] (18) at (1.475, -1.475) {};
		\node [style=none] (19) at (1.475, 0.475) {};
		\node [style=none] (20) at (0.525, 0.475) {};
		\node [style=none] (21) at (-1, 0) {$1$};
		\node [style=none] (22) at (0, -1) {$2$};
		\node [style=none] (23) at (0, 0) {$3$};
		\node [style=none] (24) at (1, 1) {$4$};
	\end{pgfonlayer}
	\begin{pgfonlayer}{edgelayer}
		\draw (11.center) to (10.center);
		\draw (10.center) to (5.center);
		\draw (5.center) to (4.center);
		\draw (12.center) to (4.center);
		\draw (4.center) to (3.center);
		\draw (3.center) to (0.center);
		\draw (0.center) to (1.center);
		\draw (1.center) to (2.center);
		\draw (2.center) to (12.center);
		\draw (2.center) to (9.center);
		\draw (9.center) to (10.center);
		\draw (5.center) to (8.center);
		\draw (8.center) to (7.center);
		\draw (7.center) to (6.center);
		\draw (6.center) to (5.center);
		\draw [style=fill grey no line] (14.center)
			 to (13.center)
			 to (16.center)
			 to (15.center)
			 to cycle;
		\draw [style=fill grey no line] (18.center)
			 to (17.center)
			 to (20.center)
			 to (19.center)
			 to cycle;
	\end{pgfonlayer}
\end{tikzpicture}$} \quad = \quad \left( 12324, \quad \scalebox{0.7}{$\begin{tikzpicture}
	\begin{pgfonlayer}{nodelayer}
		\node [style=circle vertex] (0) at (-1, 0) {$1$};
		\node [style=circle vertex] (1) at (1, 0) {$3$};
		\node [style=circle vertex] (2) at (1, 1.5) {$4$};
		\node [style=circle vertex] (3) at (1, -1.5) {$2$};
	\end{pgfonlayer}
	\begin{pgfonlayer}{edgelayer}
		\draw (3) to (0);
		\draw (0) to (1);
		\draw (2) to (0);
	\end{pgfonlayer}
\end{tikzpicture}$} \quad \right) \quad \text{ and } \quad \scalebox{0.7}{$\begin{tikzpicture}
	\begin{pgfonlayer}{nodelayer}
		\node [style=none] (0) at (-2, 1.5) {};
		\node [style=none] (1) at (-2, -1.5) {};
		\node [style=none] (2) at (-1, -1.5) {};
		\node [style=none] (3) at (-1, 1.5) {};
		\node [style=none] (4) at (2, -1.5) {};
		\node [style=none] (5) at (2, -0.5) {};
		\node [style=none] (6) at (-1, -0.5) {};
		\node [style=none] (7) at (0, 0.5) {};
		\node [style=none] (8) at (-1, 0.5) {};
		\node [style=none] (9) at (0, -0.5) {};
		\node [style=none] (10) at (0, 1.5) {};
		\node [style=none] (11) at (1, 1.5) {};
		\node [style=none] (12) at (1, 0.5) {};
		\node [style=none] (13) at (1, -0.5) {};
		\node [style=none] (14) at (2, 0.5) {};
		\node [style=none] (15) at (0.125, -0.6) {};
		\node [style=none] (16) at (1.875, -0.6) {};
		\node [style=none] (17) at (-1.5, 0) {$1$};
		\node [style=none] (18) at (0.5, -1) {$2$};
		\node [style=none] (19) at (-0.5, 0) {$3$};
		\node [style=none] (20) at (0.5, 1) {$4$};
		\node [style=none] (21) at (1.5, 0) {$5$};
		\node [style=none] (22) at (-0.975, 0.525) {};
		\node [style=none] (23) at (-0.025, 0.525) {};
		\node [style=none] (24) at (-0.025, 1.475) {};
		\node [style=none] (25) at (-0.975, 1.475) {};
		\node [style=none] (26) at (1.025, 0.525) {};
		\node [style=none] (27) at (1.975, 0.525) {};
		\node [style=none] (28) at (1.975, 1.475) {};
		\node [style=none] (29) at (1.025, 1.475) {};
		\node [style=none] (30) at (0.025, -0.475) {};
		\node [style=none] (31) at (0.975, -0.475) {};
		\node [style=none] (32) at (0.975, 0.475) {};
		\node [style=none] (33) at (0.025, 0.475) {};
	\end{pgfonlayer}
	\begin{pgfonlayer}{edgelayer}
		\draw (14.center) to (5.center);
		\draw (5.center) to (13.center);
		\draw (13.center) to (12.center);
		\draw (12.center) to (14.center);
		\draw (11.center) to (12.center);
		\draw (12.center) to (7.center);
		\draw (7.center) to (10.center);
		\draw (10.center) to (11.center);
		\draw (0.center) to (3.center);
		\draw (0.center) to (1.center);
		\draw (1.center) to (2.center);
		\draw (2.center) to (6.center);
		\draw (6.center) to (8.center);
		\draw (8.center) to (3.center);
		\draw (8.center) to (7.center);
		\draw (7.center) to (9.center);
		\draw (9.center) to (6.center);
		\draw (9.center) to (13.center);
		\draw (4.center) to (5.center);
		\draw (4.center) to (2.center);
		\draw (16.center) to (15.center);
		\draw [style=fill grey no line] (23.center)
			 to (22.center)
			 to (25.center)
			 to (24.center)
			 to cycle;
		\draw [style=fill grey no line] (27.center)
			 to (26.center)
			 to (29.center)
			 to (28.center)
			 to cycle;
		\draw [style=fill grey no line] (31.center)
			 to (30.center)
			 to (33.center)
			 to (32.center)
			 to cycle;
	\end{pgfonlayer}
\end{tikzpicture}$} \quad = \quad \left( 1232452, \quad \scalebox{0.7}{$\begin{tikzpicture}
	\begin{pgfonlayer}{nodelayer}
		\node [style=circle vertex] (0) at (-2, 0) {$1$};
		\node [style=circle vertex] (1) at (0, 0) {$3$};
		\node [style=circle vertex] (2) at (0, 1.5) {$4$};
		\node [style=circle vertex] (3) at (0, -1.5) {$2$};
		\node [style=circle vertex] (4) at (2, 0) {$5$};
	\end{pgfonlayer}
	\begin{pgfonlayer}{edgelayer}
		\draw (3) to (0);
		\draw (0) to (1);
		\draw (2) to (0);
		\draw (1) to (4);
	\end{pgfonlayer}
\end{tikzpicture}$} \quad \right)$$
\end{vbn}

\subsection{The homology of $\Quilt$} \label{parparhomology}

 We have a morphism of dg operads
\[p:\Quilt \hookrightarrow \FSurj \otimes_H \Brace \to \Com \otimes_H \Brace = \Brace
\]
sending $(W,T) \in \FWord(n) \times \Tree(n)$ to $T$ if $|W| = n$, and $0$ otherwise. The following theorem computes the homology of $\Quilt$, thus proving Hawkins' conjecture \cite[Conj. 3.7]{hawkins}.

\begin{theorem}\label{thmhomologyQuilt}
Over the integers, the morphism $p: \Quilt \longrightarrow \Brace$ is a quasi-isomorphism.
\end{theorem}

As a first step, observe that the differential of $\Quilt$ solely involves the differential of $\FSurj$. For a tree $T$, let $\FWord(T)$ be the set of words quilting $T$ and let $\Quilt(T)$ be the dg submodule of $\FSurj$ spanned by $\FWord(T)$. We have an isomorphism of chain complexes
\begin{align}
\label{Quiltsplits}
\Quilt(n) \cong \bigoplus_{T \in \Tree(n)} \Quilt(T).
\end{align}

\subsubsection{The double complex $\Quilt(T)_{\bullet,\bullet}$} \label{parparpardoublecomplex}
Fix a tree $T \in \Tree(n)$ and we can assume its vertices are in standard order. When we draw the tree as part of a quilt, then $n$ is its bottommost leaf. We assign to each $W$ that quilts $T$ a bidegree $(\deg_n(W), \deg_{\neg n}(W))$ as follows
\begin{align*}
\deg_n(W) &:= 1 - \# \text{ occurrences of }n\text{ in }W,\\
\deg_{\neg n}(W) &:= \deg(W) - \deg_n(W) = n-|W| - \deg_n(W)
\end{align*}
Hence, $\Quilt(T)_{\bullet,\bullet}$ is a bigraded complex, concentrated in the third quadrant, whose differential $\partial$ splits as $\partial_n + \partial_{\neg n}$ where 
$$\partial_n(W) = \sum_{i:W(i)=n}\sgn_W(i) \partial_i(W) \quad\text{ and }\quad \partial_{\neg n}(W) = \partial(W) - \partial_n(W) = \sum_{i: W(i)\neq n} \sgn_W(i) \partial_i(W).$$
\begin{lemma}
$\Quilt(T)_{\bullet,\bullet}$ is a double complex.
\end{lemma}
 \begin{proof}
 The equations $\partial_n^2 =0$ and $\partial_n \partial_{\neg n} + \partial_{\neg n}\partial = 0$ follow from the relations: for $j\geq i$, we have $\partial_i \partial_j = \partial_{j+1}\partial_i$ and $\sgn_W(i)\sgn_{\partial_i W}(j) = - \sgn_{\partial_{j+1}W}(i)\sgn_W(j+1)$. 
 \end{proof}
 
\subsubsection{The homology of $\Quilt(T)$} \label{parparparquiltT}
 Consider the tree $T_{\neg n} \in \Tree(n-1)$ by removing the vertex $n$ from the tree $T$, i.e. its bottommost leaf.
\begin{lemma}
Removing all occurrences of $n$ induces a surjection $\red_n:\FWord(T) \longrightarrow \FWord(T_{\neg n})$.
\end{lemma} 
\begin{proof}
First, we verify that $\red_n(W) \in \FWord$ for $W \in \FWord$. It suffices to only check the nondegeneracy condition: suppose $\red_n(W)= \ldots uu \ldots$ for some number $u$, then $W = \ldots unu \ldots$ and thus $u<_W n$ whence $n \tre_T u$. As $n$ is the bottommost leaf, nondegeneracy thus cannot occur in $\red_n(W)$.

Next, we verify that $\red_n(W)$ quilts $T_{\neg n}$ if $W$ quilts $T$. This follows from the following observation: for $u,v<n$, we have that $u <_{\red_n(W)} v$ if and only if $u <_W v$, and $u \tre_{\red_n(W)} v$ if and only if $u \tre_W v$. 

Finally, for $W \in \FWord(T_{\neg n})$, we have that $W_n \in \FWord(T)$ and $\red_n(Wn)=W$. 
\end{proof}
  For a word $W \in \FWord(T_{\neg n})$, let $\Quilt_{\bullet}^W(T)$ be the subcomplex of $(\Quilt_{\bullet,\deg(W)}(T),\partial_n)$ spanned by words $W'$ such that $\red_n(W')= W$. Notice that $\Quilt_{\bullet}^W(T)$ lives in degrees $\leq 0$. Observe that $\deg(W)$ and $\deg_n(W')$ determine the bidegree of $W'$ since $\deg(W')= \deg(W)+\deg_n(W')$. We have an isomorphism of chain complexes
  \begin{align}
  \label{QuiltTsplits} (\Quilt_{\bullet,\deg(W)}(T), \partial_n) \cong \bigoplus_{W \in \FWord(T_{\neg n})} (\Quilt_{\bullet}^W(T),  \partial_n).
\end{align}   
 
We have a unique description of every word $W'\in \FWord(T)$ that reduces to $W$.

\begin{lemma}\label{lemdecompositionword}
Given $W \in \FWord(T_{\neg n})$, there is a unique decomposition into subwords $W=W_1 \ldots W_l$ such that any number appears in exactly one $W_i$.
Furthermore, elements in $\red_n^{-1}(W)$ are obtained by inserting $n$ after certain $W_i$, i.e.
$$ \red_n^{-1}(W) = \left\{\; W_1 \ldots W_{i_1} n \ldots  W_{i_j}n \ldots W_{i_{k}} nW_{i_{k}+1} \ldots W_{l} \; \middle| \; 1 \leq i_1< \ldots <i_k \leq l \text{ and } k\geq 1 \; \right\} $$
\end{lemma}
\begin{proof}
	Let $u_1 \tre_W \ldots \tre_W u_l$ be the $<_W$-minimal numbers amongst $\{1,\ldots,n-1\}$ and let $W_i$ be the subword of $W$ starting with the first occurrence of $u_i$ and the ending with the last occurrence of $u_i$. Due to no interlacing, the words $W_1,\ldots,W_l$ are disjoint. Moreover, again due to no interlacing, all occurrences of a number $u$ in $W$ occur in a single $W_i$, namely for $i$ such that $u_i <_W u$. Hence, $W = W_1 \ldots W_l$. For example, the word $152563436787$ decomposes as $W_1W_2W_3W_4$ where $W_1 = 1, W_2 = 525, W_3 = 63436$ and $W_4 = 787$.

Let $W'\in \red_n^{-1}(W)$. As $n$ is the bottommost leaf of $T$ by assumption, $n$ is minimal for $<_{W'}$. Moreover, due to $\red_n(W')=W$, the corresponding unique decomposition for $W'$ is given by $W_1 \ldots W_{i_1} W'' W_{i_k+1} \ldots W_l$ for some $1 \leq i_1 < i_k \leq l$ and where $W''$ is the subword of $W'$ starting with the first occurrence of $n$ and ending with the last occurrence of $n$. 

We analyse the word $W''$ further: the $<_{W'}$-minimal numbers above $n$ are $u_{i_1+1},\ldots,u_{i_{k}}$ due to $\red_n(W')=W$. Hence, a similar reasoning tells us $W''= n W_{i_1+1} \ldots W_{i_2} n \ldots n  W_{i_{k-1}+1} \ldots W_{i_{k}}n$ for some $i_1<i_2<\ldots<i_{k}$, proving the result. Note that $k= \deg_n(W')+1$.
\end{proof}


\begin{lemma}\label{lemhomologyQuiltWT}
The homology of $\Quilt_{\bullet}^W(T)$ is free of rank one and concentrated in degree $0$. In particular, any quilt $(W',T)$ with a single $n$ and such that $\red_n(W')=W$ represents the class spanning $H_0(\Quilt_{\bullet}^W(T))$.
\end{lemma}
\begin{proof}
Lemma \ref{lemdecompositionword} provides an isomorphism
\begin{equation}\label{eq:phi}
	\phi:\Quilt_{\bullet}^W(T) \longrightarrow \mathbf C^{\mathrm{cell}}_{\bullet}(\Delta^{l-1})
\end{equation}
where $l$ is the number of subwords given in Lemma \ref{lemdecompositionword} and the right hand side is the cellular chain complex of the $(l-1)$th simplex (living in negative degrees, due to our cohomological conventions).
%
\end{proof}

\begin{prop}\label{prophomologyQuiltT}
The homology of $\Quilt(T)$ is free of rank one and concentrated in degree $0$. In particular, any quilt $(W,T)$ of degree $0$ represents the class generating $H_0(\Quilt(T))$.
\end{prop}
\begin{proof}
The double complex $\Quilt(T)_{\bullet,\bullet}$ is concentrated in the third quadrant and thus its horizontal filtration 
$$F_s \Quilt(T)_t = \bigoplus_{\substack{a+b=t \\ b\geq s} }\Quilt(T)_{a,b}$$
induces a converging spectral sequence. As $E^0_{st}= \Quilt(T)_{s,t}$ and $d^0 = \partial_n$, we obtain by \eqref{QuiltTsplits}
$$E^1_{st} = H_s(\Quilt(T)_{\bullet,t}) = \bigoplus_{\substack{W \in \FSurj(T_{\neg n}) \\ \deg(W)= t-1}} H_{s}(\Quilt_{\bullet}^W(T)).$$

By  Lemma \ref{lemhomologyQuiltWT}, if $s\ne 0$ this is zero and otherwise

$$E^1_{0t} \cong \Quilt(T_{\neg n})_{t-1}.$$
 Moreover, under these identifications, its differential $d^1= \partial_{\neg n}$ corresponds to the differential $\partial$ of $\Quilt(T_{\neg n})[1]$. As $T_{\neg n}$ has strictly fewer vertices than $T$, we obtain by induction on $n$ that $E^2$ is concentrated in degree $0$, free of rank one and living only in the first filtration piece.
{Due to convergence, $E^2$ computes the homology of $\Quilt(T)$.}
\end{proof}

\begin{proof}[Proof of Theorem \ref{thmhomologyQuilt}]
We verify that the morphism of dg operads $p:\Quilt \longrightarrow \Brace$ is a quasi-isomorphism. As $\Brace$ is a dg operad concentrated in degree $0$ with trivial differential, it suffices to show that $H_s(\Quilt)=0$ for $s\neq 0$ and that $H(p):H_0(\Quilt) \longrightarrow \Brace$ is an isomorphism. By Proposition \ref{prophomologyQuiltT} and \eqref{Quiltsplits}, the first condition holds. Furthermore, they show that $H_0(\Quilt(n))\cong \bigoplus_{T \in \Tree(n)}k$ and that moreover for every $T \in \Tree(n)$, the unique generating class can be represented by any quilt $(W,T)\in \Quilt(n)$ of degree $0$. Hence, the projection $p$ induces an isomorphism $H_0(\Quilt(n)) \cong \Brace(n)$.   
\end{proof}

\subsection{The morphism $\Prelieinf \longrightarrow \Quilt$}\label{parparprelie}

We show that the morphism $\Linf\longrightarrow \Quilt$ established in \cite[Thm. 7.8]{hawkins} factors through the the operad $\Prelieinf$, the minimal model of the Koszul operad $\PreLie$, lifting the morphism $\PreLie \longrightarrow \Brace$ in homology.
\subsubsection{The operad $\Prelieinf$}

The operad $\PreLie$ is Koszul with Koszul dual operad $\Perm$, which is $n$-dimensional in arity $n$ \cite[Prop. 2.1]{chapotonlivernet2001}. As a result, its minimal model $\Prelieinf$ is generated by the operations 
$$pl_n \in \Prelieinf(n) \text{ of degree } 2-n $$ 

such that $pl_n^\si  =(-1)^\si pl_n$ for $\si \in \Ss_n$ such that $\si(1)=1$, and with differential 
\begin{equation*}
{\partial(pl_n) = \sum_{\substack{k+l = n+1 \\ k,l \geq 2 }} \sum_{\substack{\chi \in \Sh_{l,k-1} \\ \chi(1) =1 }}(-1)^{l(k-1)+ \chi+1} (pl_k \circ_1 pl_l)^{\chi^{-1}} + \sum_{\substack{\chi \in \Sh_{l+1,k-2} \\ \chi(1) =1 }}\sum_{j=1,\ldots,l} (-1)^{kl+\chi+(1j)+1}  \left(pl_k \circ_2 pl_l^{(1j)}\right)^{\chi^{-1}} \label{prelierel}}
\end{equation*}
We have a morphism $\Linf \longrightarrow \Prelieinf$ sending $l_n$ to $\sum_{j=1}^n {(-1)^{(1j)}} pl_n^{(1j)}=pl_n - \sum_{j=2}^n pl_n^{(1j)}$.

\subsubsection{The morphism $\Prelieinf \longrightarrow \Quilt$}\label{parparparPreLietoQuilt}
For $n\geq 2$, the morphism $\Linf \longrightarrow \Quilt$ sends $l_n$ to the operations $L_n$ which are defined as the antisymmetrization of operations $P_n$, i.e $$L_n := \sum_{\si \in \Ss_n} (-1)^\si P_n^\si \quad \text{ where } \quad P_n := \sum_{\substack{ Q \in \mathrm{Quilt}(n) \\ \deg(Q)= 2-n \\ Q \text{ in standard order}}} (-1)^{1+ \frac{n(n-1)}{2}} Q$$
Note that we have reversed the degrees of both $\Linf$ and $\Quilt$ with respect to \cite{hawkins} and \cite{vanhermanslowen2022}. In \cite[Ex. 4.7]{vanhermanslowen2022}, the quilts making up $L_2,L_3$ and $L_4$ are drawn.

\begin{mydef}
For $n\geq 2$, define the degree $2-n$ operations 
$$ PL_n := \sum_{\substack{\si \in \Ss_n \\ \si(1) = 1 }} (-1)^\si P_n^\si \in \Quilt(n).$$
\end{mydef}

\begin{prop}\label{propprelieinf}
We have a morphism 
$$\Prelieinf \longrightarrow \Quilt$$
sending $pl_n$ to $PL_n$.
\end{prop}
\begin{proof}
Unraveling the relation \eqref{prelierel} for the operations $PL_n$, we aim to show the equation
\begin{equation}\label{prelieinfquilt}
\sum_{\substack{\si \in \Ss_n \\ \si(1) = 1 }} (-1)^\si \partial (P_n^\si) = \sum_{\substack{\si \in \Ss_n \\ \si(1) = 1  }} \sum_{\substack{ k+l =n+1 \\ k,l \geq 2 \\ i=1,\ldots,k}}  (-1)^\si (-1)^{(k-1)l+(i-1)(l-1)}(P_k \circ_i P_l)^\si.
\end{equation}
The proof of \cite[Thm. 7.8]{hawkins} consists of showing that for each quilt appearing in either $\partial(P_n)$ or $P_k \circ_i P_l$ for $k+l=n+1$ and index $i$, there is a unique counterpart in either $\partial(P_n^\si)$ or $(P_{k'} \circ_{i'} P_{l'})^{\si}$ for unique numbers $k',l',i'$ and unique permutation $\si \in \Ss_n$. We observe that for quilts $Q$ and $Q'$ in standard order, the labelling of the root of the quilts appearing in either $Q \circ_i Q'$ and $\partial(Q)$ remains unchanged, that is, it keeps label $1$ under these operations. Hence, this is also true for the quilts appearing in $P_k \circ_i P_l$ or $\partial(P_n)$ and thus the number $1$ is a fixpoint of the above unique permutation $\si$. We deduce that the proof of \cite[Thm. 7.8]{hawkins} descends to a proof of equation \eqref{prelieinfquilt}.
\end{proof}

\section{A new model for $\mathsf{E}_2$}\label{parmodel}
%
%
%

\subsection{Twisting of $\Brace$} \label{parpartwistbrace}
Recall that we have a morphism $\Lie \longrightarrow \Brace$ given by $l_2:=C_2 - C_2^{(12)}$, the antisymmetrisation of the $2$-corolla. We apply the twisting procedure for operads as in \cite[\S 5.5]{dotsenkoshadrinvallette2024}. {From now on, we work over a field of characteristic zero, as the twisting formalism is not defined otherwise.}
\begin{mydef}\label{deftwbrace}
Let $\TwBrace$ be the twisting of $\Brace$ by a MC-element, i.e. 
$$\TwBrace = (\Brace \vee m, \partial^m)$$
the coproduct of $\Brace$ with a formal element $m$ of arity $0$ and degree $1$, with differential 
\begin{align*}
\partial^m(m) &= \frac{1}{2} l_2(m,m)  \\
\partial^m(T) &= l_2(m,T) - \sum_{j=1}^m T \circ_j l_2(m,-).
\end{align*}  

for $T \in \Brace(m)$.
\end{mydef}


Following \cite[\S 9]{dolgushevwillwacher2015} and \cite[Prop. 5.23]{dotsenkoshadrinvallette2024}, we provide a $k$-module basis of $\TwBrace(n)$ as follows. Let a \emph{tree with black vertices} $(T,I)$ consist of a tree $T\in \Tree(n+n')$ and a subset $I$ of $\{1,\ldots,n+n'\}$ of cardinality $n'$. The tuple can be drawn as a tree $T$ of $n+n'$ vertices such that each vertex $i\in I$ is coloured black. {Vertices in $I$ are considered indistinguishable (or unlabeled), whereas the remaining $n$ vertices preserve their linear order, or equivalently, are labeled from $1$ to $n$.} These correspond to the elements of $\TwBrace(n)$ consisting of a tree $T$ such that each vertex $i\in I$ is filled by an instance of $m$ through composition. Observe that composition of trees with black vertices is simply composing their underlying trees and colouring the correct vertices black.

\begin{mydef}\label{deftwbracebar}
Let $\TwBraceover$ be the graded suboperad of $\TwBrace$ spanned by the trees with black vertices such that each black vertex has at least two children.
\end{mydef}
\begin{opm}
In contrast with $\TwBrace$, $\TwBraceover$ is finite dimensional in all arities. In \cite[\S 9]{dolgushevwillwacher2015}, $\TwBraceover$ is denoted $\Br$.
\end{opm}
The following is a combination of results on the surjection operad \cite{bergerfresse2004} with theorems \cite[Prop. 9.2, Thm 9.3]{dolgushevwillwacher2015}.
\begin{prop}
$\TwBraceover$ is an $\mathsf{E}_2$-suboperad of $\TwBrace$. Furthermore, $\TwBraceover$ is isomorphic to $\FSurj$.
\end{prop}
\begin{opm}\label{rem:historical remark}
The dg operads $\TwBraceover$ and $\FSurj$ are both isomorphic to the dg operad $\HG$ which encodes homotopy G-algebras. Historically, Gerstenhaber and Voronov were the first to establish the Gerstenhaber up to homotopy structure on the Hochschild cochain complex and coined the term ``homotopy G-algebra'' \cite{gerstenhabervoronov}. The corresponding dg operad $\HG$ is generated by an associative binary operation $m_2 \in \HG(2)$ of degree $0$ (``a dot product'') and, for every $n\geq 1$, an element $B_n \in \HG(n+1)$ of degree $-n$ (``a $n$-brace'') satisfying the homotopy G-algebra relations (see \cite[Example 4.1]{KVZ} or \cite[\S 3.2]{voronov}). The isomorphism $\HG \cong \FSurj$ realizes the generators $m_2$ and $B_n$ as the words $12$ and $121\ldots1n1$ \cite[\S 1.6.4]{bergerfresse2004}. The isomorphism $\HG \cong \TwBraceover$ realises the generators $m_2$ and $B_n$ as the following trees with black vertices
$$ \scalebox{0.7}{$\begin{tikzpicture}
	\begin{pgfonlayer}{nodelayer}
		\node [style=full black] (0) at (0, -1) {};
		\node [style=circle vertex] (1) at (-1.5, 1) {$1$};
		\node [style=circle vertex] (2) at (1.5, 1) {$2$};
	\end{pgfonlayer}
	\begin{pgfonlayer}{edgelayer}
		\draw (0) to (1);
		\draw (0) to (2);
	\end{pgfonlayer}
\end{tikzpicture}
$} \quad \text{ and } \quad \scalebox{0.7}{$\begin{tikzpicture}
	\begin{pgfonlayer}{nodelayer}
		\node [style=circle vertex] (0) at (0, -1) {$1$};
		\node [style=circle vertex] (1) at (-2, 1) {$2$};
		\node [style=circle vertex] (2) at (2, 1) {$n$};
		\node [style=none] (3) at (0, 1) {$\ldots$};
	\end{pgfonlayer}
	\begin{pgfonlayer}{edgelayer}
		\draw (0) to (1);
		\draw (0) to (2);
	\end{pgfonlayer}
\end{tikzpicture}
$}\quad .$$

For a proof, see \cite[Cor. 29]{willwacher2016}, keeping in mind that Willwacher considers the $A_\infty$ version of these operads (see also \cite[Rem. 26]{willwacher2016}).

\end{opm}

%
%

\subsection{Twisting of $\Quilt$} \label{parpartwistquilt}

Recall from \S \ref{parparparPreLietoQuilt} that we have a morphism of operads $\Linf \longrightarrow \Quilt$. We apply the twisting procedure from \cite[\S 5.5]{dotsenkoshadrinvallette2024}.

\begin{mydef}
Let $\TwQuilt$ be the twisting of $\Quilt$ by a MC-element, i.e. 
$$\TwQuilt = (\Quilt \hat{\vee} \al, \partial^\alpha)$$
the completed coproduct of $\Quilt$ with a formal element $\al$ of arity $0$ and degree $1$, with differential
\begin{align*}
\partial^{\alpha}(\al) &= \sum_{n\geq 2}\frac{(n-1)(-1)^{\frac{n(n+1)}{2}+1}}{n!} L_n(\al,\ldots,\al), \\
\partial^{\al}(Q) &= \partial(Q) + \sum_{n\geq 2} \frac{(-1)^{\frac{n(n+1)}{2}+1}}{(n-1)!} L_n(\al,\ldots,\al,Q) + \sum_{j=1}^m \frac{(-1)^{\deg(Q)+\frac{n(n+1)}{2}}}{(n-1)!} Q \circ_j L_n(\al,\ldots,\al,-)
\end{align*}
for $Q \in \Quilt(m)$.
\end{mydef}

Quilts with black rectangles provide a $k$-module basis of $\TwQuilt(n)$ as follows. Let a \emph{quilt with black rectangles} $(Q,I)$ consist of a quilt $Q\in \Quilt(n+n')$ and a subset $I$ of $\{1,\ldots,n+n'\}$ of cardinality $n'$. The tuple can be drawn as a quilt $Q$ of $n+n'$ rectangles such that each rectangle $i\in I$ is coloured black. These correspond to the elements of $\TwQuilt$ consisting of a quilt $Q$ such that each rectangle $i\in I$ is filled by an instance of $\al$ through composition. Observe that composition of quilts with black rectangles is simply composing their underlying quilts and colouring the correct rectangles black.

{Notice that the quasi-isomorphism $p: \Quilt \longrightarrow \Brace$ from Theorem \ref{thmhomologyQuilt} is compatible with the respective maps from the $\Linf$ operad. The following result is therefore an immediate consequence of Theorem \cite[Thm. 5.1]{dolgushevwillwacher2015}.} 

\begin{theorem}\label{thmmodel}
The twisted projection
$$\Tw(p): \TwQuilt \longrightarrow \TwBrace$$
which applies the projection on quilts and sends $\alpha$ to $m$, is a quasi-isomorphism.
\end{theorem}

\begin{cor}\label{cormodel}
The operad $\TwQuilt$ is an $\mathsf{E}_2$-operad.
\end{cor}

\begin{opm}
Recall that  there is a quasi-isomorphism $\mathsf{Lie}\stackrel{\sim}\to \mathsf{TwLie}$ while there is no map $\PreLie$ to $\mathsf{TwPreLie}$, only a quasi-isomorphism $\mathsf{Lie}\to \mathsf{TwPreLie}$ \cite{dotsenko2024homotopical}, which evokes the fact that twisting a pre-Lie algebra is not generally a dg pre-Lie algebra. 

In particular, in the light of Proposition \ref{propprelieinf} we see that the $\Linf$ structure on $\TwQuilt$-algebras, actually arises from twisting a $\PreLie_\infty$-algebra structure. 
	\end{opm}

\section{The $\mathsf{E}_2$-action on the Gerstenhaber-Schack complex} \label{paraction}

\subsection{The Gerstenhaber-Schack complex for prestacks}\label{parparGS}

We recall the notions of prestack and its associated Gerstenhaber-Schack complex, thus fixing terminology and notations. We use the same terminology as in \cite{DVL}, \cite{lowenvandenberghCCT}.

A prestack is a pseudofunctor taking values in $k$-linear categories. Let $\U$ be a small category.

\begin{mydef}
A prestack $\A= (\A,m,f,c)$ over $\U$ consists of the following data:
\begin{itemize}
\item for every object $U \in \U$, a $k$-linear category $(\A(U), m^U, 1^U)$ where $m^U$ is the composition of morphisms in $\A(U)$ and $1^U$ encodes the identity morphisms of $\A(U)$.
\item for every morphism $u:V \longrightarrow U$ in $\U$, a $k$-linear functor $f^u = u \st : \A(U) \longrightarrow \A(V)$. For $u= 1_U$ the identity morphism of $U$ in $\U$,  we require that $(1_U)\st = 1_{\A(U)}$.
\item for every couple of morphisms $v: W \longrightarrow V, u: V \longrightarrow U$ in $\U$, a natural isomorphism
$$c^{u,v}: v\st u \st \longrightarrow (uv)\st.$$
For $u =1$ or $v=1$, we require $c^{u,v} =1$. Moreover, the natural isomorphisms have to satisfy the following coherence condition for every triple $w:T \longrightarrow W$, $v :W \longrightarrow V$ and $u:V \longrightarrow U$:
$$c^{u,vw}(c^{v,w} \circ 	u \st) = c^{uv,w}(w\st \circ c^{u,v}).$$
The data $(m,f,c)$ are also called the \emph{multiplications}, \emph{restrictions} and \emph{twists} of $\A$ respectively. 
\end{itemize}
A presheaf of categories $\A$ is a prestack for which all twists are trivial, i.e. $c^{u,v}=1$ for every $u$ and $v$.
\end{mydef}

Given such a prestack $\A$, we have an associated Gerstenhaber--Schack complex $\CGS(\A)$. In \cite{DVL} this is defined as the totalisation of a multicomplex $\textbf{C}^{\bu,\bu}(\A)$. We first review some notations.

\medskip

\noindent \emph{Notations.}
Let $\si = (U_{0} \overset{u_{1}}{\rightarrow} U_{1} \rightarrow  \ldots   \overset{u_{p}}{\rightarrow}  U_{p} )$ be a $p$-simplex in the category $\U$, then we have two functors $\A(U_{p}) \longrightarrow \A(U_{0})$, namely
$$\si\hs := u_{1}\st  \ldots  u_{p} \st \quad\text{ and }\quad \si\st := (u_{p}\ldots  u_{1})\st$$
For each $1 \leq k \leq p-1$, define the subsimplices $
L_{k}(\si) = (U_{0} \overset{u_{1}}{\rightarrow} U_{1} \rightarrow  \ldots   \overset{u_{k}}{\rightarrow}  U_{k}) \text{ and } R_{k}(\si) = (U_{k} \overset{u_{k+1}}{\rightarrow} U_{k+1} \rightarrow  \ldots   \overset{u_{p}}{\rightarrow}  U_{p})$ and the natural isomorphism $c^{\si,k} = c^{u_{k}\ldots u_{1},u_{p}\ldots u_{k+1}}:(L_{k}\si)\st(R_{k}(\si))\st \longrightarrow \si \st$.

\begin{mydef}
Let $p,q\geq 0$, then define
$$\textbf{C}^{p,q}(\A) = \prod_{\si \in N_{p}(\U)} \prod_{A \in \A(U_{p})^{q+1}} \Hom\left(\bigotimes_{i=1}^{q} \A(U_{p})(A_{i},A_{i-1}), \A(U_{0})(\si\hs A_{q},\si\st A_{0})\right)$$
for $N(\U)$ the nerve of $\U$, and set
$$\CGS^{n}(\A) = \bigoplus_{p+q=n}\textbf{C}^{p,q}(\A)$$
The GS complex is a multicomplex with differential $d=\sum_{j=0}^{q+1} d_j$ for $d_{j}: \textbf{C}^{p,q}(\A) \longrightarrow \textbf{C}^{p+j,q+1-j}(\A)$. We provide a definition of $d_0$ and $d_1$, respectively called the \emph{Hochschild} and \emph{simplicial} component, below, and refer to \cite[Def. 3.2]{vanhermanslowen2022} for a detailed description of $d_j$ for $j\geq 2$.
\end{mydef}

Elements of the GS complex have a neat geometric interpretation as rectangles: for $\te \in \textbf{C}^{p,q}(\A)$ and the data $(\si,A,a)$ from above, we can represent $\te^\si(A)(a)$ as the rectangle of data
$$\begin{tikzpicture}
	\begin{pgfonlayer}{nodelayer}
		\node [style=none] (0) at (0.4, 0) {};
		\node [style=none] (1) at (1.6, 0) {};
		\node [style=none] (2) at (2.15, 0) {$A_1$};
		\node [style=none] (3) at (0, 0) {$A_0$};
		\node [style=none] (4) at (2.7, 0) {};
		\node [style=none] (7) at (5.9, 0) {};
		\node [style=none] (8) at (7.1, 0) {};
		\node [style=none] (9) at (7.5, 0) {$A_q$};
		\node [style=none] (10) at (5.15, 0) {$A_{q-1}$};
		\node [style=none] (11) at (4.175, 0) {};
		\node [style=none] (12) at (1, 0.5) {$a_1$};
		\node [style=none] (13) at (6.65, 0.5) {$a_q$};
		\node [style=none] (14) at (7.5, -6.35) {};
		\node [style=none] (15) at (7.5, -5.15) {};
		\node [style=none] (16) at (7.5, -4.5) {$u_2^{*}\ldots u_p^{*} A_q$};
		\node [style=none] (17) at (7.5, -6.75) {$\sigma^{\#} A_q$};
		\node [style=none] (18) at (7.5, -4.025) {};
		\node [style=none] (19) at (7.5, -1.6) {};
		\node [style=none] (20) at (7.5, -0.4) {};
		\node [style=none] (22) at (7.5, -2) {$u_p^{*}A_q$};
		\node [style=none] (23) at (7.5, -2.55) {};
		\node [style=none] (24) at (8, -5.75) {$u_1^{*}$};
		\node [style=none] (25) at (8, -0.85) {$u_p^{*}$};
		\node [style=none] (26) at (0, -0.5) {};
		\node [style=none] (27) at (0, -6.25) {};
		\node [style=none] (28) at (1, -6.75) {};
		\node [style=none] (29) at (6.5, -6.75) {};
		\node [style=none] (30) at (0, -6.75) {$\sigma^{*}A_0$};
		\node [style=none] (31) at (-0.725, -3.75) {$\sigma^{*}$};
		\node [style=none] (32) at (4, -7.5) {$\theta^{\sigma}(A)(a_1,\ldots,a_q)$};
		\node [style=none] (33) at (3.75, -3.25) {$\theta^{\sigma}(A)$};
	\end{pgfonlayer}
	\begin{pgfonlayer}{edgelayer}
		\draw [style=arrow] (1.center) to (0.center);
		\draw [style=arrow] (8.center) to (7.center);
		\draw [style=arrow] (15.center) to (14.center);
		\draw [style=arrow] (20.center) to (19.center);
		\draw [style=arrow] (29.center) to (28.center);
		\draw [style=arrow] (26.center) to (27.center);
		\draw [style=dashed arrow] (23.center) to (18.center);
		\draw [style=dashed arrow] (11.center) to (4.center);
	\end{pgfonlayer}
\end{tikzpicture}$$
In particular, the prestack data $(m,f,c) \in \CGS^2(\A)$ can be depicted as 
$$\begin{tikzpicture}
	\begin{pgfonlayer}{nodelayer}
		\node [style=none] (0) at (0, 0) {$m^U$};
		\node [style=none] (1) at (-1.7, 1) {};
		\node [style=none] (2) at (-0.55, 1) {};
		\node [style=none] (3) at (0, 1) {$A_1$};
		\node [style=none] (4) at (-2.15, 1) {$A_0$};
		\node [style=none] (5) at (0.625, 1) {};
		\node [style=none] (6) at (1.775, 1) {};
		\node [style=none] (7) at (2.325, 1) {$A_2$};
		\node [style=none] (8) at (-1.7, -1) {};
		\node [style=none] (11) at (-2.15, -1) {$A_0$};
		\node [style=none] (13) at (1.775, -1) {};
		\node [style=none] (14) at (2.325, -1) {$A_2$};
		\node [style=none] (15) at (-2.25, 0.525) {};
		\node [style=none] (16) at (-2, 0.525) {};
		\node [style=none] (17) at (-2.25, -0.475) {};
		\node [style=none] (18) at (-2, -0.475) {};
		\node [style=none] (19) at (2.225, 0.525) {};
		\node [style=none] (20) at (2.475, 0.525) {};
		\node [style=none] (21) at (2.225, -0.475) {};
		\node [style=none] (22) at (2.475, -0.475) {};
		\node [style=none] (23) at (-2.85, 0) {$1_U$};
		\node [style=none] (24) at (2.975, 0) {$1_U$};
	\end{pgfonlayer}
	\begin{pgfonlayer}{edgelayer}
		\draw [style=arrow] (2.center) to (1.center);
		\draw [style=arrow] (6.center) to (5.center);
		\draw (18.center) to (16.center);
		\draw (15.center) to (17.center);
		\draw (22.center) to (20.center);
		\draw (19.center) to (21.center);
		\draw [style=arrow] (13.center) to (8.center);
	\end{pgfonlayer}
\end{tikzpicture} \qquad \begin{tikzpicture}
	\begin{pgfonlayer}{nodelayer}
		\node [style=none] (0) at (0, 0) {$f^u$};
		\node [style=none] (1) at (-0.625, 1) {};
		\node [style=none] (2) at (0.525, 1) {};
		\node [style=none] (4) at (-1.075, 1) {$A_0$};
		\node [style=none] (7) at (1.05, 1) {$A_1$};
		\node [style=none] (8) at (-0.625, -1) {};
		\node [style=none] (9) at (-1.575, -1) {$u^{*}A_0$};
		\node [style=none] (10) at (0.5, -1) {};
		\node [style=none] (11) at (1.55, -1) {$u^{*}A_1$};
		\node [style=none] (12) at (-1.075, 0.525) {};
		\node [style=none] (14) at (-1.075, -0.475) {};
		\node [style=none] (17) at (1.05, 0.525) {};
		\node [style=none] (19) at (1.05, -0.475) {};
		\node [style=none] (20) at (-1.675, 0) {$u^{*}$};
		\node [style=none] (21) at (1.725, 0) {$u^{*}$};
	\end{pgfonlayer}
	\begin{pgfonlayer}{edgelayer}
		\draw [style=arrow] (2.center) to (1.center);
		\draw [style=arrow] (10.center) to (8.center);
		\draw [style=arrow] (17.center) to (19.center);
		\draw [style=arrow] (12.center) to (14.center);
	\end{pgfonlayer}
\end{tikzpicture} \qquad \begin{tikzpicture}
	\begin{pgfonlayer}{nodelayer}
		\node [style=none] (0) at (0.025, 0) {$c^{u,v}$};
		\node [style=none] (1) at (-0.575, 1.9) {};
		\node [style=none] (2) at (0.55, 1.9) {};
		\node [style=none] (3) at (-1.05, 2) {$A_0$};
		\node [style=none] (4) at (1.075, 2) {$A_0$};
		\node [style=none] (5) at (-0.65, -2) {};
		\node [style=none] (6) at (-2.05, -2) {$(uv)^{*}A_0$};
		\node [style=none] (7) at (0.525, -2) {};
		\node [style=none] (8) at (1.825, -2) {$v^{*}u^{*}A_0$};
		\node [style=none] (9) at (1.075, -0.475) {};
		\node [style=none] (10) at (1.075, -1.475) {};
		\node [style=none] (11) at (-1.05, 1.525) {};
		\node [style=none] (12) at (-1.05, -1.475) {};
		\node [style=none] (13) at (1.675, -1) {$v^{*}$};
		\node [style=none] (14) at (-1.875, 0) {$(uv)^{*}$};
		\node [style=none] (15) at (1.075, 1.525) {};
		\node [style=none] (16) at (1.075, 0.525) {};
		\node [style=none] (17) at (1.675, 1) {$u^{*}$};
		\node [style=none] (18) at (1.675, 0.025) {$u^* A_0$};
		\node [style=none] (19) at (-0.575, 2.05) {};
		\node [style=none] (20) at (0.55, 2.05) {};
	\end{pgfonlayer}
	\begin{pgfonlayer}{edgelayer}
		\draw (2.center) to (1.center);
		\draw [style=arrow] (7.center) to (5.center);
		\draw [style=arrow] (11.center) to (12.center);
		\draw [style=arrow] (9.center) to (10.center);
		\draw [style=arrow] (15.center) to (16.center);
		\draw (20.center) to (19.center);
	\end{pgfonlayer}
\end{tikzpicture} $$
Similarly, we can draw different components of the differential $d$ using rectangles. For the Hochschild component $d_0(\te)^\si(A)$ we have
\begin{align*}
\scalebox{0.9}{$\begin{tikzpicture}
	\begin{pgfonlayer}{nodelayer}
		\node [style=none] (0) at (-3.075, 2.925) {};
		\node [style=none] (1) at (-1.875, 2.925) {};
		\node [style=none] (4) at (0.05, 2.925) {};
		\node [style=none] (5) at (1.75, 2.925) {};
		\node [style=none] (6) at (2.95, 2.925) {};
		\node [style=none] (9) at (1.525, 2.925) {};
		\node [style=none] (10) at (-2.475, 3.425) {$a_1$};
		\node [style=none] (11) at (2.35, 3.425) {$a_{q+1}$};
		\node [style=none] (12) at (3.1, -1.75) {};
		\node [style=none] (13) at (3.1, -0.55) {};
		\node [style=none] (16) at (3.1, -0.375) {};
		\node [style=none] (17) at (3.1, 1.6) {};
		\node [style=none] (18) at (3.1, 2.8) {};
		\node [style=none] (20) at (3.1, 1.3) {};
		\node [style=none] (21) at (3.6, -1.15) {$u_1^{*}$};
		\node [style=none] (22) at (3.6, 2.2) {$u_p^{*}$};
		\node [style=none] (23) at (-3.175, 2.825) {};
		\node [style=none] (24) at (-3.175, -1.725) {};
		\node [style=none] (25) at (-1.425, -1.875) {};
		\node [style=none] (26) at (2.85, -1.875) {};
		\node [style=none] (29) at (0.575, 0.6) {$\theta^{\sigma}(A)$};
		\node [style=none] (31) at (-1.65, 2.825) {};
		\node [style=none] (32) at (-1.65, -1.725) {};
		\node [style=none] (33) at (-2.975, -1.875) {};
		\node [style=none] (34) at (-1.775, -1.875) {};
		\node [style=none] (35) at (-1.5, 2.925) {};
		\node [style=none] (36) at (-0.3, 2.925) {};
		\node [style=none] (37) at (-0.875, 3.425) {$a_2$};
		\node [style=none] (38) at (-3, -2.95) {};
		\node [style=none] (39) at (2.925, -2.95) {};
		\node [style=none] (40) at (-3.225, -2.025) {};
		\node [style=none] (41) at (-3.1, -2.025) {};
		\node [style=none] (42) at (-3.225, -2.85) {};
		\node [style=none] (43) at (-3.1, -2.85) {};
		\node [style=none] (44) at (3.05, -2.025) {};
		\node [style=none] (45) at (3.175, -2.025) {};
		\node [style=none] (46) at (3.05, -2.85) {};
		\node [style=none] (47) at (3.175, -2.85) {};
		\node [style=none] (48) at (3.2, -2.95) {};
		\node [style=none] (49) at (0.075, -2.45) {$m^{U_0}$};
		\node [style=none] (50) at (-2.3, 0.425) {$f^{\sigma}$};
		\node [style=none] (93) at (-3.725, 0.425) {$\sigma^*$};
	\end{pgfonlayer}
	\begin{pgfonlayer}{edgelayer}
		\draw [style=arrow] (1.center) to (0.center);
		\draw [style=arrow] (6.center) to (5.center);
		\draw [style=arrow] (13.center) to (12.center);
		\draw [style=arrow] (18.center) to (17.center);
		\draw [style=arrow] (26.center) to (25.center);
		\draw [style=arrow] (23.center) to (24.center);
		\draw [style=dashed arrow] (20.center) to (16.center);
		\draw [style=dashed arrow] (9.center) to (4.center);
		\draw [style=arrow] (31.center) to (32.center);
		\draw [style=arrow] (34.center) to (33.center);
		\draw [style=arrow] (36.center) to (35.center);
		\draw [style={right_arrow}, in=180, out=0] (38.center) to (39.center);
		\draw (43.center) to (41.center);
		\draw (40.center) to (42.center);
		\draw (47.center) to (45.center);
		\draw (44.center) to (46.center);
	\end{pgfonlayer}
\end{tikzpicture}$} \; + \; \sum_{i=1}^q (-1)^i \; \scalebox{0.9}{$\begin{tikzpicture}
	\begin{pgfonlayer}{nodelayer}
		\node [style=none] (0) at (-3.175, 1.925) {};
		\node [style=none] (1) at (3.125, 1.925) {};
		\node [style=none] (2) at (4.325, 1.925) {};
		\node [style=none] (3) at (-1.7, 1.925) {};
		\node [style=none] (4) at (4.475, -2.75) {};
		\node [style=none] (5) at (4.475, -1.55) {};
		\node [style=none] (6) at (4.475, -1.375) {};
		\node [style=none] (7) at (4.475, 0.6) {};
		\node [style=none] (8) at (4.475, 1.8) {};
		\node [style=none] (9) at (4.475, 0.3) {};
		\node [style=none] (10) at (-4.825, 1.825) {};
		\node [style=none] (11) at (-4.825, -2.725) {};
		\node [style=none] (12) at (-4.65, -2.875) {};
		\node [style=none] (13) at (4.225, -2.875) {};
		\node [style=none] (14) at (0, -0.425) {$\theta^{\sigma}(A)$};
		\node [style=none] (15) at (-4.725, 1.925) {};
		\node [style=none] (16) at (-3.525, 1.925) {};
		\node [style=none] (17) at (-1.375, 2.925) {};
		\node [style=none] (18) at (-0.175, 2.925) {};
		\node [style=none] (19) at (-1.425, 1.925) {};
		\node [style=none] (20) at (1.225, 1.925) {};
		\node [style=none] (21) at (-1.575, 2.85) {};
		\node [style=none] (22) at (-1.45, 2.85) {};
		\node [style=none] (23) at (-1.575, 2.025) {};
		\node [style=none] (24) at (-1.45, 2.025) {};
		\node [style=none] (25) at (1.225, 2.85) {};
		\node [style=none] (26) at (1.35, 2.85) {};
		\node [style=none] (27) at (1.225, 2.025) {};
		\node [style=none] (28) at (1.35, 2.025) {};
		\node [style=none] (29) at (1.375, 1.925) {};
		\node [style=none] (30) at (-0.025, 2.925) {};
		\node [style=none] (31) at (1.175, 2.925) {};
		\node [style=none] (32) at (1.475, 1.925) {};
		\node [style=none] (33) at (2.95, 1.925) {};
		\node [style=none] (34) at (0, 2.375) {$m^{U_p}$};
		\node [style=none] (35) at (4.975, -2.15) {$u_1^{*}$};
		\node [style=none] (36) at (4.975, 1.175) {$u_p^{*}$};
		\node [style=none] (37) at (-4.025, 2.425) {$a_1$};
		\node [style=none] (38) at (3.8, 2.425) {$a_{q+1}$};
		\node [style=none] (39) at (-0.825, 3.425) {$a_i$};
		\node [style=none] (40) at (0.575, 3.425) {$a_{i+1}$};
		\node [style=none] (41) at (-5.425, -0.575) {$\sigma^*$};
	\end{pgfonlayer}
	\begin{pgfonlayer}{edgelayer}
		\draw [style=arrow] (2.center) to (1.center);
		\draw [style=arrow] (5.center) to (4.center);
		\draw [style=arrow] (8.center) to (7.center);
		\draw [style=arrow] (13.center) to (12.center);
		\draw [style=arrow] (10.center) to (11.center);
		\draw [style=dashed arrow] (9.center) to (6.center);
		\draw [style=dashed arrow] (3.center) to (0.center);
		\draw [style=arrow] (16.center) to (15.center);
		\draw [style=arrow] (18.center) to (17.center);
		\draw (19.center) to (20.center);
		\draw (24.center) to (22.center);
		\draw (21.center) to (23.center);
		\draw (28.center) to (26.center);
		\draw (25.center) to (27.center);
		\draw [style=arrow] (31.center) to (30.center);
		\draw [style=dashed arrow] (33.center) to (32.center);
	\end{pgfonlayer}
\end{tikzpicture}$} \; + (-1)^{q+1} \; \scalebox{0.9}{$\begin{tikzpicture}
	\begin{pgfonlayer}{nodelayer}
		\node [style=none] (0) at (1.75, 2.925) {};
		\node [style=none] (1) at (3.2, 2.925) {};
		\node [style=none] (2) at (-1.45, 2.925) {};
		\node [style=none] (3) at (0.25, 2.925) {};
		\node [style=none] (4) at (1.45, 2.925) {};
		\node [style=none] (5) at (0.025, 2.925) {};
		\node [style=none] (6) at (-2.4, 3.425) {$a_1$};
		\node [style=none] (7) at (2.425, 3.425) {$a_{q+1}$};
		\node [style=none] (8) at (1.6, -1.75) {};
		\node [style=none] (9) at (1.6, -0.55) {};
		\node [style=none] (10) at (1.6, -0.375) {};
		\node [style=none] (11) at (1.6, 1.6) {};
		\node [style=none] (12) at (1.6, 2.8) {};
		\node [style=none] (13) at (1.6, 1.3) {};
		\node [style=none] (14) at (-3.1, 2.825) {};
		\node [style=none] (15) at (-3.1, -1.725) {};
		\node [style=none] (16) at (-2.925, -1.875) {};
		\node [style=none] (17) at (1.35, -1.875) {};
		\node [style=none] (18) at (-0.925, 0.6) {$\theta^{\sigma}(A)$};
		\node [style=none] (19) at (1.85, -1.875) {};
		\node [style=none] (20) at (3.3, -1.875) {};
		\node [style=none] (21) at (-3, 2.925) {};
		\node [style=none] (22) at (-1.8, 2.925) {};
		\node [style=none] (23) at (-0.8, 3.425) {$a_2$};
		\node [style=none] (24) at (-2.925, -2.95) {};
		\node [style=none] (25) at (3.25, -2.95) {};
		\node [style=none] (26) at (-3.15, -2.025) {};
		\node [style=none] (27) at (-3.025, -2.025) {};
		\node [style=none] (28) at (-3.15, -2.85) {};
		\node [style=none] (29) at (-3.025, -2.85) {};
		\node [style=none] (30) at (3.375, -2.025) {};
		\node [style=none] (31) at (3.5, -2.025) {};
		\node [style=none] (32) at (3.375, -2.85) {};
		\node [style=none] (33) at (3.5, -2.85) {};
		\node [style=none] (34) at (3.525, -2.95) {};
		\node [style=none] (35) at (0.15, -2.45) {$m^{U_0}$};
		\node [style=none] (36) at (-3.725, 0.425) {$\sigma^*$};
		\node [style=none] (38) at (3.425, -1.75) {};
		\node [style=none] (39) at (3.425, -0.55) {};
		\node [style=none] (40) at (3.425, -0.375) {};
		\node [style=none] (41) at (3.425, 1.6) {};
		\node [style=none] (42) at (3.425, 2.8) {};
		\node [style=none] (43) at (3.425, 1.3) {};
		\node [style=none] (44) at (3.925, -1.15) {$u_1^{*}$};
		\node [style=none] (45) at (3.925, 2.175) {$u_p^{*}$};
		\node [style=none] (46) at (2.575, -1.2) {$f^{u_1}$};
		\node [style=none] (47) at (1.725, 1.425) {};
		\node [style=none] (48) at (3.175, 1.425) {};
		\node [style=none] (49) at (1.7, -0.45) {};
		\node [style=none] (50) at (3.15, -0.45) {};
		\node [style=none] (51) at (2.625, 2.175) {$f^{u_p}$};
	\end{pgfonlayer}
	\begin{pgfonlayer}{edgelayer}
		\draw [style=arrow] (1.center) to (0.center);
		\draw [style=arrow] (4.center) to (3.center);
		\draw [style=arrow] (9.center) to (8.center);
		\draw [style=arrow] (12.center) to (11.center);
		\draw [style=arrow] (17.center) to (16.center);
		\draw [style=arrow] (14.center) to (15.center);
		\draw [style=dashed arrow] (13.center) to (10.center);
		\draw [style=dashed arrow] (5.center) to (2.center);
		\draw [style=arrow] (20.center) to (19.center);
		\draw [style=arrow] (22.center) to (21.center);
		\draw [style={right_arrow}, in=180, out=0] (24.center) to (25.center);
		\draw (29.center) to (27.center);
		\draw (26.center) to (28.center);
		\draw (33.center) to (31.center);
		\draw (30.center) to (32.center);
		\draw [style=arrow] (39.center) to (38.center);
		\draw [style=arrow] (42.center) to (41.center);
		\draw [style=dashed arrow] (43.center) to (40.center);
		\draw [style=arrow] (48.center) to (47.center);
		\draw [style=arrow] (50.center) to (49.center);
	\end{pgfonlayer}
\end{tikzpicture}$}
\end{align*}
Note that $d_0$ constitutes a differential as well, i.e. it squares to $0$. The simplicial component $(-1)^{p+q+1} d_1(\te)^\si(A)$ can similarly be drawn as
\begin{align*}
\scalebox{0.9}{$\begin{tikzpicture}
	\begin{pgfonlayer}{nodelayer}
		\node [style=none] (2) at (0.65, 3.375) {};
		\node [style=none] (3) at (1.85, 3.375) {};
		\node [style=none] (4) at (2.8, 3.375) {};
		\node [style=none] (5) at (1.625, 3.375) {};
		\node [style=none] (6) at (-0.05, 3.875) {$a_1$};
		\node [style=none] (7) at (2.275, 3.875) {$a_q$};
		\node [style=none] (8) at (2.95, -1.3) {};
		\node [style=none] (9) at (2.95, -0.1) {};
		\node [style=none] (10) at (2.95, 0.075) {};
		\node [style=none] (11) at (2.95, 2.05) {};
		\node [style=none] (12) at (2.95, 3.25) {};
		\node [style=none] (13) at (2.95, 1.75) {};
		\node [style=none] (14) at (-0.75, 3.275) {};
		\node [style=none] (15) at (-0.75, -1.275) {};
		\node [style=none] (16) at (-0.575, -1.425) {};
		\node [style=none] (17) at (2.7, -1.425) {};
		\node [style=none] (18) at (1.175, 1.05) {$\theta^{\partial_0\sigma}(A)$};
		\node [style=none] (21) at (-0.65, 3.375) {};
		\node [style=none] (22) at (0.3, 3.375) {};
		\node [style=none] (24) at (-2.825, -4) {};
		\node [style=none] (25) at (2.85, -4) {};
		\node [style=none] (26) at (-3.05, -3.075) {};
		\node [style=none] (27) at (-2.925, -3.075) {};
		\node [style=none] (28) at (-3.05, -3.9) {};
		\node [style=none] (29) at (-2.925, -3.9) {};
		\node [style=none] (30) at (2.975, -3.075) {};
		\node [style=none] (31) at (3.1, -3.075) {};
		\node [style=none] (32) at (2.975, -3.9) {};
		\node [style=none] (33) at (3.1, -3.9) {};
		\node [style=none] (34) at (3.125, -4) {};
		\node [style=none] (35) at (0.75, -3.5) {$m^{U_0}$};
		\node [style=none] (36) at (-1.8, 0.125) {$c^{\sigma,1,A_0}$};
		\node [style=none] (44) at (3.525, -0.7) {$u_2^{*}$};
		\node [style=none] (45) at (3.825, 2.625) {$u_{p+1}^{*}$};
		\node [style=none] (46) at (3, -4.25) {};
		\node [style=none] (47) at (2.95, -2.775) {};
		\node [style=none] (48) at (2.95, -1.575) {};
		\node [style=none] (49) at (-0.575, -2.9) {};
		\node [style=none] (50) at (2.7, -2.9) {};
		\node [style=none] (51) at (-0.725, -2.775) {};
		\node [style=none] (52) at (-0.725, -1.575) {};
		\node [style=none] (53) at (-2.875, 3.375) {};
		\node [style=none] (54) at (-0.925, 3.375) {};
		\node [style=none] (55) at (-2.85, 3.25) {};
		\node [style=none] (56) at (-0.9, 3.25) {};
		\node [style=none] (57) at (-2.975, 3.275) {};
		\node [style=none] (58) at (-2.975, -2.825) {};
		\node [style=none] (59) at (-2.825, -2.9) {};
		\node [style=none] (60) at (-0.875, -2.9) {};
		\node [style=none] (61) at (1.375, -2.15) {$f^{u_1}$};
		\node [style=none] (62) at (3.55, -2.15) {$u_1^*$};
		\node [style=none] (63) at (-3.475, 0.1) {$\sigma^*$};
		\node [style=none] (64) at (-1.9, 3.75) {$A_0$};
	\end{pgfonlayer}
	\begin{pgfonlayer}{edgelayer}
		\draw [style=arrow] (4.center) to (3.center);
		\draw [style=arrow] (9.center) to (8.center);
		\draw [style=arrow] (12.center) to (11.center);
		\draw [style=arrow] (17.center) to (16.center);
		\draw [style=arrow] (14.center) to (15.center);
		\draw [style=dashed arrow] (13.center) to (10.center);
		\draw [style=dashed arrow] (5.center) to (2.center);
		\draw [style=arrow] (22.center) to (21.center);
		\draw [style={right_arrow}, in=180, out=0] (24.center) to (25.center);
		\draw (29.center) to (27.center);
		\draw (26.center) to (28.center);
		\draw (33.center) to (31.center);
		\draw (30.center) to (32.center);
		\draw [style=arrow] (48.center) to (47.center);
		\draw [style=arrow] (50.center) to (49.center);
		\draw [style=arrow] (52.center) to (51.center);
		\draw (54.center) to (53.center);
		\draw (56.center) to (55.center);
		\draw [style=arrow] (57.center) to (58.center);
		\draw [style=arrow] (60.center) to (59.center);
	\end{pgfonlayer}
\end{tikzpicture}$} \; + \; \sum_{i=1}^p (-1)^{i} \; \scalebox{0.9}{$\begin{tikzpicture}
	\begin{pgfonlayer}{nodelayer}
		\node [style=none] (0) at (-2.825, 4.55) {};
		\node [style=none] (1) at (-1.625, 4.55) {};
		\node [style=none] (2) at (-0.675, 4.55) {};
		\node [style=none] (3) at (-1.85, 4.55) {};
		\node [style=none] (4) at (-3.525, 5.125) {$a_1$};
		\node [style=none] (5) at (-1.2, 5.125) {$a_q$};
		\node [style=none] (6) at (-0.525, -3.125) {};
		\node [style=none] (7) at (-0.525, -1.925) {};
		\node [style=none] (8) at (-0.525, 2.175) {};
		\node [style=none] (9) at (-0.525, 3.225) {};
		\node [style=none] (10) at (-0.525, 4.425) {};
		\node [style=none] (11) at (-0.525, 3.1) {};
		\node [style=none] (12) at (-4.225, 4.45) {};
		\node [style=none] (13) at (-4.225, -3.1) {};
		\node [style=none] (14) at (-4.05, -3.25) {};
		\node [style=none] (15) at (-0.775, -3.25) {};
		\node [style=none] (16) at (-2.35, 0.725) {$\theta^{\partial_i\sigma}(A)$};
		\node [style=none] (17) at (-4.125, 4.55) {};
		\node [style=none] (18) at (-3.175, 4.55) {};
		\node [style=none] (19) at (-4.05, -4.25) {};
		\node [style=none] (20) at (3.55, -4.25) {};
		\node [style=none] (21) at (-4.275, -3.325) {};
		\node [style=none] (22) at (-4.15, -3.325) {};
		\node [style=none] (23) at (-4.275, -4.15) {};
		\node [style=none] (24) at (-4.15, -4.15) {};
		\node [style=none] (25) at (3.675, -3.325) {};
		\node [style=none] (26) at (3.8, -3.325) {};
		\node [style=none] (27) at (3.675, -4.15) {};
		\node [style=none] (28) at (3.8, -4.15) {};
		\node [style=none] (29) at (3.825, -4.25) {};
		\node [style=none] (30) at (-1.225, -3.75) {$m^{U_0}$};
		\node [style=none] (31) at (1.575, 0.625) {$c^{u_i,u_{i+1}}(A_q)$};
		\node [style=none] (32) at (4.35, -2.525) {$u_1^{*}$};
		\node [style=none] (33) at (4.775, 3.8) {$u_{p+1}^{*}$};
		\node [style=none] (34) at (-0.425, 2.125) {};
		\node [style=none] (35) at (3.525, 2.125) {};
		\node [style=none] (36) at (-0.4, 2) {};
		\node [style=none] (37) at (3.55, 2) {};
		\node [style=none] (38) at (-0.525, 2.025) {};
		\node [style=none] (39) at (-0.525, -0.7) {};
		\node [style=none] (40) at (-0.375, -0.775) {};
		\node [style=none] (41) at (3.575, -0.775) {};
		\node [style=none] (42) at (-4.7, 0.6) {$\sigma^*$};
		\node [style=none] (43) at (3.75, 0.775) {};
		\node [style=none] (44) at (3.75, 1.975) {};
		\node [style=none] (45) at (4.4, -0.1) {$u_i^{*}$};
		\node [style=none] (46) at (3.75, -0.675) {};
		\node [style=none] (47) at (3.75, 0.525) {};
		\node [style=none] (48) at (4.625, 1.35) {$u_{i+1}^{*}$};
		\node [style=none] (49) at (-0.525, -1.8) {};
		\node [style=none] (50) at (-0.525, -0.875) {};
		\node [style=none] (51) at (-0.45, -3.225) {};
		\node [style=none] (52) at (3.675, -3.225) {};
		\node [style=none] (53) at (3.75, -3.125) {};
		\node [style=none] (54) at (3.75, -1.925) {};
		\node [style=none] (55) at (3.75, -1.8) {};
		\node [style=none] (56) at (3.75, -0.875) {};
		\node [style=none] (57) at (3.75, 2.175) {};
		\node [style=none] (58) at (3.75, 3.225) {};
		\node [style=none] (59) at (3.75, 4.425) {};
		\node [style=none] (60) at (3.75, 3.1) {};
		\node [style=none] (61) at (-0.375, 4.625) {};
		\node [style=none] (62) at (3.575, 4.625) {};
		\node [style=none] (63) at (-0.35, 4.5) {};
		\node [style=none] (64) at (3.6, 4.5) {};
		\node [style=none] (65) at (1.6, 3.3) {$R_{i+1}\sigma^{\#}$};
		\node [style=none] (66) at (1.55, -1.95) {$L_{i-1}\sigma^{\#}$};
		\node [style=none] (67) at (1.5, 5) {$A_q$};
	\end{pgfonlayer}
	\begin{pgfonlayer}{edgelayer}
		\draw [style=arrow] (2.center) to (1.center);
		\draw [style=arrow] (7.center) to (6.center);
		\draw [style=arrow] (10.center) to (9.center);
		\draw [style=arrow] (15.center) to (14.center);
		\draw [style=arrow] (12.center) to (13.center);
		\draw [style=dashed arrow] (11.center) to (8.center);
		\draw [style=dashed arrow] (3.center) to (0.center);
		\draw [style=arrow] (18.center) to (17.center);
		\draw [style={right_arrow}, in=180, out=0] (19.center) to (20.center);
		\draw (24.center) to (22.center);
		\draw (21.center) to (23.center);
		\draw (28.center) to (26.center);
		\draw (25.center) to (27.center);
		\draw (35.center) to (34.center);
		\draw (37.center) to (36.center);
		\draw [style=arrow] (38.center) to (39.center);
		\draw [style=arrow] (41.center) to (40.center);
		\draw [style=arrow] (44.center) to (43.center);
		\draw [style=arrow] (47.center) to (46.center);
		\draw [style=dashed arrow] (50.center) to (49.center);
		\draw [style=arrow] (52.center) to (51.center);
		\draw [style=arrow] (54.center) to (53.center);
		\draw [style=dashed arrow] (56.center) to (55.center);
		\draw [style=arrow] (59.center) to (58.center);
		\draw [style=dashed arrow] (60.center) to (57.center);
		\draw (62.center) to (61.center);
		\draw (64.center) to (63.center);
	\end{pgfonlayer}
\end{tikzpicture}$}  + \; (-1)^{p+1} \;\scalebox{0.9}{$\begin{tikzpicture}
	\begin{pgfonlayer}{nodelayer}
		\node [style=none] (0) at (0.15, 1.625) {};
		\node [style=none] (1) at (1.85, 1.625) {};
		\node [style=none] (2) at (3.325, 1.625) {};
		\node [style=none] (3) at (1.625, 1.625) {};
		\node [style=none] (4) at (-0.975, 3.6) {$a_1$};
		\node [style=none] (5) at (2.8, 3.575) {$a_q$};
		\node [style=none] (6) at (3.475, -3.05) {};
		\node [style=none] (7) at (3.475, -1.85) {};
		\node [style=none] (8) at (3.475, -1.675) {};
		\node [style=none] (9) at (3.475, 0.3) {};
		\node [style=none] (10) at (3.475, 1.5) {};
		\node [style=none] (11) at (3.475, 0) {};
		\node [style=none] (12) at (-1.725, 1.525) {};
		\node [style=none] (13) at (-1.725, -3.025) {};
		\node [style=none] (14) at (-1.55, -3.175) {};
		\node [style=none] (15) at (3.225, -3.175) {};
		\node [style=none] (16) at (0.675, -0.7) {$\theta^{\partial_0\sigma}(A)$};
		\node [style=none] (17) at (-1.625, 1.625) {};
		\node [style=none] (18) at (-0.2, 1.625) {};
		\node [style=none] (19) at (-3.8, -4.25) {};
		\node [style=none] (20) at (3.375, -4.25) {};
		\node [style=none] (21) at (-4.025, -3.325) {};
		\node [style=none] (22) at (-3.9, -3.325) {};
		\node [style=none] (23) at (-4.025, -4.15) {};
		\node [style=none] (24) at (-3.9, -4.15) {};
		\node [style=none] (25) at (3.5, -3.325) {};
		\node [style=none] (26) at (3.625, -3.325) {};
		\node [style=none] (27) at (3.5, -4.15) {};
		\node [style=none] (28) at (3.625, -4.15) {};
		\node [style=none] (29) at (3.65, -4.25) {};
		\node [style=none] (30) at (0.25, -3.75) {$m^{U_0}$};
		\node [style=none] (31) at (-2.775, -0.125) {$c^{\sigma,p,A_0}$};
		\node [style=none] (32) at (4.05, 0.925) {$u_p^{*}$};
		\node [style=none] (33) at (4.3, 2.275) {$u_{p+1}^{*}$};
		\node [style=none] (34) at (3.525, -4.5) {};
		\node [style=none] (35) at (0.025, 1.675) {};
		\node [style=none] (36) at (0.025, 2.875) {};
		\node [style=none] (37) at (-1.575, 3.05) {};
		\node [style=none] (38) at (-0.075, 3.05) {};
		\node [style=none] (39) at (-1.725, 1.675) {};
		\node [style=none] (40) at (-1.725, 2.875) {};
		\node [style=none] (41) at (-3.85, 3.125) {};
		\node [style=none] (42) at (-1.9, 3.125) {};
		\node [style=none] (43) at (-3.825, 3) {};
		\node [style=none] (44) at (-1.875, 3) {};
		\node [style=none] (45) at (-3.95, 3.025) {};
		\node [style=none] (46) at (-3.95, -3.075) {};
		\node [style=none] (47) at (-3.8, -3.15) {};
		\node [style=none] (48) at (-1.85, -3.15) {};
		\node [style=none] (49) at (-0.825, 2.3) {$f^{u_{p+1}}$};
		\node [style=none] (50) at (4.075, -2.4) {$u_1^*$};
		\node [style=none] (51) at (3.475, 1.675) {};
		\node [style=none] (52) at (3.475, 2.875) {};
		\node [style=none] (53) at (1.825, 3.05) {};
		\node [style=none] (54) at (3.375, 3.05) {};
		\node [style=none] (55) at (2.675, 2.3) {$f^{u_{p+1}}$};
		\node [style=none] (56) at (1.75, 1.675) {};
		\node [style=none] (57) at (1.75, 2.875) {};
		\node [style=none] (58) at (0.175, 3.025) {};
		\node [style=none] (59) at (1.65, 3.025) {};
		\node [style=none] (60) at (-4.5, 0) {$\sigma^*$};
	\end{pgfonlayer}
	\begin{pgfonlayer}{edgelayer}
		\draw [style=arrow] (2.center) to (1.center);
		\draw [style=arrow] (7.center) to (6.center);
		\draw [style=arrow] (10.center) to (9.center);
		\draw [style=arrow] (15.center) to (14.center);
		\draw [style=arrow] (12.center) to (13.center);
		\draw [style=dashed arrow] (11.center) to (8.center);
		\draw [style=dashed arrow] (3.center) to (0.center);
		\draw [style=arrow] (18.center) to (17.center);
		\draw [style={right_arrow}, in=180, out=0] (19.center) to (20.center);
		\draw (24.center) to (22.center);
		\draw (21.center) to (23.center);
		\draw (28.center) to (26.center);
		\draw (25.center) to (27.center);
		\draw [style=arrow] (36.center) to (35.center);
		\draw [style=arrow] (38.center) to (37.center);
		\draw [style=arrow] (40.center) to (39.center);
		\draw (42.center) to (41.center);
		\draw (44.center) to (43.center);
		\draw [style=arrow] (45.center) to (46.center);
		\draw [style=arrow] (48.center) to (47.center);
		\draw [style=arrow] (52.center) to (51.center);
		\draw [style=arrow] (54.center) to (53.center);
		\draw [style=arrow] (57.center) to (56.center);
		\draw [style=dashed arrow] (59.center) to (58.center);
	\end{pgfonlayer}
\end{tikzpicture}$}
\end{align*}
In case $\A$ is a presheaf of categories, $d_0+d_1$ defines a differential making $(\CC^{\bullet,\bullet}(\A),d_0+d_1)$ a bicomplex. This is the original complex devised by Gerstenhaber and Schack \cite{gerstenhaberschack1}.

We will also be interested in the subcomplex $\CGSnr(\A) \sub \CGS(\A)$ of normalized and reduced cochains which is shown to be quasi-isomorphic to the GS complex \cite[Prop. 3.16]{DVL}. Moreover, on normalized and reduced chains, the differentials $d$ and $d_0+d_1$ coincide. A simplex $\si= (\fromto{u}{p})$ is \emph{reduced} if $u_i = 1_{U_i}$ for some $1 \leq i \leq p$. A cochain $\te = \left(\te^{\si}(A)\right)_{\si,A} \in \CGS(\A)$ is \emph{reduced} if $\te^\si(A) = 0$ for every reduced simplex $\si$. A simplex $a=(\fromto{a}{q})$ in $\A(U)$ is \emph{normal} if $a_i= 1^{U}$ for some $1 \leq i \leq q$. A cochain $\te$ is \emph{normalized} if $\te^{\si}(A)(a) = 0$ for every normal simplex $a$ in $\A(U_p)$.

\begin{opm} \label{remhomotopytransfer}
In \cite{DVL}, Dinh Van and Lowen realise the (normalized and reduced) GS complex $\CGS(\A)$ as a deformation retract of the Hochschild complex $\textbf{C}_{\mathbb{U}}(\tilde{A})$ of the associated $\mathbb{U}$-graded category $\tilde{A}$.
As a result, homotopy transfer equips the GS complex with an algebra structure over any cofibrant solution $\mathcal{Q} \sim \Disk$ of the classical Deligne conjecture. However, the chain maps and homotopy involved in the deformation retract turn out to be intricate and complex, which is already reflected by the fact that they each have an infinite amount of components. Furthermore, as the dg operad encoding Homotopy G-algebras $\HG$ is not cofibrant, finding an explicit cofibrant resolution $\mathcal{Q} \overset{\sim}{\longrightarrow} \HG$ that is a solution to the classical Deligne conjecture is not easy. For example, the evident solution provided by the minimal model $\Ginf$ of the operad $\Gerst$ encoding Gerstenhaber-algebras is cofibrant, yet the morphism $\Ginf \longrightarrow \HG$ provided by Tamarkin \cite{tamarkin1998} is inexplicit.
\end{opm}

\subsection{The action of $\TwQuilt$ for prestacks} \label{parparaction}

\subsubsection{The action of $\Quilt$ for prestacks}
\label{parparparactionquilt}

In \cite[\S 3]{vanhermanslowen2022}, the authors construct a morphism of dg operads
$$ \psi: \Quilt \longrightarrow \End(s^{-1}\CGS(\A),d_0)$$
Remark, this action holds only with respect to the Hochschild differential $d_0$.

Let us describe this action intuitively: a quilt $Q$ acts on GS cochains $(\te_1,\ldots,\te_n)$ via $\psi$ by interpreting them as rectangles (see \S \ref{parparGS}) and composing them according to $Q$, filling in possible `open spaces' by instances of restrictions $f$ and composing with multiplications $m$ both at the bottom and wherever a double line is drawn in $Q$. For a detailed description we refer to \cite[\S 3]{vanhermanslowen2022}. 

Here, we will make the action more concrete using examples.

\begin{vb}\label{ex:action}
The quilt on the right from Examples \ref{exquilt} acts on the cochains
$$
\underline{\te}=(\te_1,\te_2,\te_3,\te_4,\te_5) \in  \CC^{3,1}(\A)\oplus  \CC^{1,3}(\A) \oplus \CC^{2,2}(\A) \oplus \CC^{2,1}(\A) \oplus \CC^{1,1}(\A)$$ given the simplex $\si=(u_1,\ldots,u_5) \in N(\U)$ as
\begin{align*}
\psi\left(\; \scalebox{0.7}{$\begin{tikzpicture}
	\begin{pgfonlayer}{nodelayer}
		\node [style=none] (0) at (-2, 1.5) {};
		\node [style=none] (1) at (-2, -1.5) {};
		\node [style=none] (2) at (-1, -1.5) {};
		\node [style=none] (3) at (-1, 1.5) {};
		\node [style=none] (4) at (2, -1.5) {};
		\node [style=none] (5) at (2, -0.5) {};
		\node [style=none] (6) at (-1, -0.5) {};
		\node [style=none] (7) at (0, 0.5) {};
		\node [style=none] (8) at (-1, 0.5) {};
		\node [style=none] (9) at (0, -0.5) {};
		\node [style=none] (10) at (0, 1.5) {};
		\node [style=none] (11) at (1, 1.5) {};
		\node [style=none] (12) at (1, 0.5) {};
		\node [style=none] (13) at (1, -0.5) {};
		\node [style=none] (14) at (2, 0.5) {};
		\node [style=none] (15) at (0.125, -0.6) {};
		\node [style=none] (16) at (1.875, -0.6) {};
		\node [style=none] (17) at (-1.5, 0) {$1$};
		\node [style=none] (18) at (0.5, -1) {$2$};
		\node [style=none] (19) at (-0.5, 0) {$3$};
		\node [style=none] (20) at (0.5, 1) {$4$};
		\node [style=none] (21) at (1.5, 0) {$5$};
		\node [style=none] (22) at (-0.975, 0.525) {};
		\node [style=none] (23) at (-0.025, 0.525) {};
		\node [style=none] (24) at (-0.025, 1.475) {};
		\node [style=none] (25) at (-0.975, 1.475) {};
		\node [style=none] (26) at (1.025, 0.525) {};
		\node [style=none] (27) at (1.975, 0.525) {};
		\node [style=none] (28) at (1.975, 1.475) {};
		\node [style=none] (29) at (1.025, 1.475) {};
		\node [style=none] (30) at (0.025, -0.475) {};
		\node [style=none] (31) at (0.975, -0.475) {};
		\node [style=none] (32) at (0.975, 0.475) {};
		\node [style=none] (33) at (0.025, 0.475) {};
	\end{pgfonlayer}
	\begin{pgfonlayer}{edgelayer}
		\draw (14.center) to (5.center);
		\draw (5.center) to (13.center);
		\draw (13.center) to (12.center);
		\draw (12.center) to (14.center);
		\draw (11.center) to (12.center);
		\draw (12.center) to (7.center);
		\draw (7.center) to (10.center);
		\draw (10.center) to (11.center);
		\draw (0.center) to (3.center);
		\draw (0.center) to (1.center);
		\draw (1.center) to (2.center);
		\draw (2.center) to (6.center);
		\draw (6.center) to (8.center);
		\draw (8.center) to (3.center);
		\draw (8.center) to (7.center);
		\draw (7.center) to (9.center);
		\draw (9.center) to (6.center);
		\draw (9.center) to (13.center);
		\draw (4.center) to (5.center);
		\draw (4.center) to (2.center);
		\draw (16.center) to (15.center);
		\draw [style=fill grey no line] (23.center)
			 to (22.center)
			 to (25.center)
			 to (24.center)
			 to cycle;
		\draw [style=fill grey no line] (27.center)
			 to (26.center)
			 to (29.center)
			 to (28.center)
			 to cycle;
		\draw [style=fill grey no line] (31.center)
			 to (30.center)
			 to (33.center)
			 to (32.center)
			 to cycle;
	\end{pgfonlayer}
\end{tikzpicture}$}  \;\right)(\underline{\te})^{\si}&= \quad \scalebox{0.9}{$\input{quiltaction1.tikz}$}  \quad + \quad   \scalebox{0.9}{$\input{quiltaction2.tikz}$}  \\
&+ \quad \scalebox{0.9}{$\input{quiltaction3.tikz}$}  \quad + \quad   \scalebox{0.9}{$\input{quiltaction4.tikz}$}  \\
&+ \quad \scalebox{0.9}{$\input{quiltaction6.tikz}$}  \quad + \quad   \scalebox{0.9}{$\input{quiltaction5.tikz}$}  
\end{align*}
where we marked the added instances of $m$ and $f$ by green.
\end{vb}

\subsubsection{The morphism $\TwQuilt \longrightarrow \End(s^{-1}\CGS(\A))$} \label{parparparactiontwistquilt}

In \cite[\S 4]{vanhermanslowen2022}, the authors extend the action of $\Quilt$ on $s^{-1}\CGS(\A)$ by including the twists $c$ and employing the operad $\Quilt_b[\![c]\!]$, which we can rephrase as follows: $\Quilt_b[\![c]\!]$ is the completed coproduct of $\Quilt$ and a formal element $c$ of arity $0$ and degree $1$ imposing the following relations:
\begin{enumerate}
\item $\partial(c) = 0$,
\item $L_2(c,c) = 0$,
\item $Q \circ_i c = 0$ if $i$ has either more than two horizontal children, or at least one horizontal child.
\end{enumerate}
The action $\psi$ of $\Quilt$ extends to $\Quilt_b[\![c]\!]$ by sending the formal element $c$ to the twist $c\in \CC^{2,0}(\A)$, obtaining a morphism 
$$\psi_c: \Quilt_b[\![c]\!] \longrightarrow \End(s^{-1}\CGS(\A)).$$
Further, they obtain a new morphism $\Linf \longrightarrow \Quilt_b[\![c]\!]:l_n \longmapsto L^c_n$ via twisting with $c$ \cite[Thm 4.10]{vanhermanslowen2022}: define for $n\geq 1$ 
$$ L_n^c := \sum_{r\geq 0} \frac{(-1)^{rn+\frac{r(r+1)}{2}}}{r!} L_{n+r}(\underbrace{c,\ldots,c}_{r\text{-times}},-,\ldots,-) .$$
The new differential is given by
$$\partial^c = \partial + \partial_{L^c_1} .$$
\begin{lemma}
We have a surjective morphism of dg operads $\TwQuilt \longrightarrow \Quilt_b[\![c]\!]$ that is the identity on quilts and sends $\alpha$ to $c$. 
\end{lemma}
\begin{proof}
It suffices to verify that the differential is preserved. For $Q \in \Quilt$, unravelling the definitions shows $\partial^\al(Q)$ is exactly $\partial^c(Q)$. Further, $\partial^c(c) = \partial_{L^c_1}(c)$ which corresponds on the nose to $\partial^\al(\al)$ when replacing $c$ by $\alpha$. 
\end{proof}
\begin{opm}
Observe that $\Quilt_b[\![c]\!]$ is a quotient of $\TwQuilt$ by the ideal spanned by the MC-equation of $\alpha$ and some extra relations on $c$.
\end{opm}

This gives us the desired solution to the Deligne conjecture for prestacks.
\begin{theorem}\label{thmaction}
The $\mathsf{E}_2$-operad $\TwQuilt$ acts on the desuspended Gerstenhaber-Schack complex, i.e. we have a morphism of dg operads
$$\TwQuilt \longrightarrow \End(s^{-1}\CGS(\A)).$$
\end{theorem}

\subsection{Another action of $\TwQuilt$ for presheaves}\label{parparactionpresheaves}

For this section, let $\A:\U \longrightarrow \Cat(k)$ be a presheaf of categories.

\subsubsection{Another action of $\Quilt$ for presheaves} \label{parparparactionpresheaves}

In \cite{hawkins}, Hawkins obtains an action of $\Quilt$ on the desuspended GS complex for presheaves in a fundamentally different way, as we now will explain. For a detailed description of this morphism 
$$\psi^{\mathrm{Hawkins}}:\Quilt \longrightarrow \End(s^{-1}\CGS(\A),d_1)$$
we refer to \cite[Def. 4.22]{hawkins}. Observe that this is a morphism of dg operads with respect to the simplicial differential of the GS complex. Note that for prestacks the simplicial component $d_1$ of the differential importantly not even squares to zero.

By switching the role of trees and words, we can interpret a quilt on its side: a tree determines the vertical adjacencies and a word determines the horizontal adjacencies. For instance, Examples \ref{exquilt} are instead drawn as 
$$ \scalebox{0.7}{$
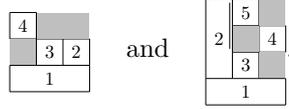
\begin{tikzpicture}
	\begin{pgfonlayer}{nodelayer}
		\node [style=none] (0) at (-1.5, -1.5) {};
		\node [style=none] (1) at (1.5, -1.5) {};
		\node [style=none] (2) at (1.5, -0.5) {};
		\node [style=none] (3) at (-1.5, -0.5) {};
		\node [style=none] (4) at (-0.5, -0.5) {};
		\node [style=none] (5) at (-0.5, 0.5) {};
		\node [style=none] (6) at (-1.5, 0.5) {};
		\node [style=none] (7) at (-1.5, 1.5) {};
		\node [style=none] (8) at (-0.5, 1.5) {};
		\node [style=none] (9) at (1.5, 0.5) {};
		\node [style=none] (10) at (0.5, 0.5) {};
		\node [style=none] (11) at (0.5, -0.5) {};
		\node [style=none] (12) at (0.5, -0.5) {};
		\node [style=none] (13) at (-0.525, -0.475) {};
		\node [style=none] (14) at (-0.525, 0.475) {};
		\node [style=none] (15) at (-1.475, 0.475) {};
		\node [style=none] (16) at (-1.475, -0.475) {};
		\node [style=none] (17) at (1.475, 0.525) {};
		\node [style=none] (18) at (1.475, 1.475) {};
		\node [style=none] (19) at (-0.475, 1.475) {};
		\node [style=none] (20) at (-0.475, 0.525) {};
		\node [style=none] (21) at (0, -1) {$1$};
		\node [style=none] (22) at (1, 0) {$2$};
		\node [style=none] (23) at (0, 0) {$3$};
		\node [style=none] (24) at (-1, 1) {$4$};
	\end{pgfonlayer}
	\begin{pgfonlayer}{edgelayer}
		\draw (11.center) to (10.center);
		\draw (10.center) to (5.center);
		\draw (5.center) to (4.center);
		\draw (12.center) to (4.center);
		\draw (4.center) to (3.center);
		\draw (3.center) to (0.center);
		\draw (0.center) to (1.center);
		\draw (1.center) to (2.center);
		\draw (2.center) to (12.center);
		\draw (2.center) to (9.center);
		\draw (9.center) to (10.center);
		\draw (5.center) to (8.center);
		\draw (8.center) to (7.center);
		\draw (7.center) to (6.center);
		\draw (6.center) to (5.center);
		\draw [style=fill grey no line] (14.center)
			 to (13.center)
			 to (16.center)
			 to (15.center)
			 to cycle;
		\draw [style=fill grey no line] (18.center)
			 to (17.center)
			 to (20.center)
			 to (19.center)
			 to cycle;
	\end{pgfonlayer}
\end{tikzpicture}$} \quad \text{ and } \quad \scalebox{0.7}{$\begin{tikzpicture}
	\begin{pgfonlayer}{nodelayer}
		\node [style=none] (0) at (1.5, -2) {};
		\node [style=none] (1) at (-1.5, -2) {};
		\node [style=none] (2) at (-1.5, -1) {};
		\node [style=none] (3) at (1.5, -1) {};
		\node [style=none] (4) at (-1.5, 2) {};
		\node [style=none] (5) at (-0.5, 2) {};
		\node [style=none] (6) at (-0.5, -1) {};
		\node [style=none] (7) at (0.5, 0) {};
		\node [style=none] (8) at (0.5, -1) {};
		\node [style=none] (9) at (-0.5, 0) {};
		\node [style=none] (10) at (1.5, 0) {};
		\node [style=none] (11) at (1.5, 1) {};
		\node [style=none] (12) at (0.5, 1) {};
		\node [style=none] (13) at (-0.5, 1) {};
		\node [style=none] (14) at (0.5, 2) {};
		\node [style=none] (15) at (-0.6, 0.125) {};
		\node [style=none] (16) at (-0.6, 1.875) {};
		\node [style=none] (17) at (0, -1.5) {$1$};
		\node [style=none] (18) at (-1, 0.5) {$2$};
		\node [style=none] (19) at (0, -0.5) {$3$};
		\node [style=none] (20) at (1, 0.5) {$4$};
		\node [style=none] (21) at (0, 1.5) {$5$};
		\node [style=none] (22) at (0.525, -0.975) {};
		\node [style=none] (23) at (0.525, -0.025) {};
		\node [style=none] (24) at (1.475, -0.025) {};
		\node [style=none] (25) at (1.475, -0.975) {};
		\node [style=none] (26) at (0.525, 1.025) {};
		\node [style=none] (27) at (0.525, 1.975) {};
		\node [style=none] (28) at (1.475, 1.975) {};
		\node [style=none] (29) at (1.475, 1.025) {};
		\node [style=none] (30) at (-0.475, 0.025) {};
		\node [style=none] (31) at (-0.475, 0.975) {};
		\node [style=none] (32) at (0.475, 0.975) {};
		\node [style=none] (33) at (0.475, 0.025) {};
	\end{pgfonlayer}
	\begin{pgfonlayer}{edgelayer}
		\draw (14.center) to (5.center);
		\draw (5.center) to (13.center);
		\draw (13.center) to (12.center);
		\draw (12.center) to (14.center);
		\draw (11.center) to (12.center);
		\draw (12.center) to (7.center);
		\draw (7.center) to (10.center);
		\draw (10.center) to (11.center);
		\draw (0.center) to (3.center);
		\draw (0.center) to (1.center);
		\draw (1.center) to (2.center);
		\draw (2.center) to (6.center);
		\draw (6.center) to (8.center);
		\draw (8.center) to (3.center);
		\draw (8.center) to (7.center);
		\draw (7.center) to (9.center);
		\draw (9.center) to (6.center);
		\draw (9.center) to (13.center);
		\draw (4.center) to (5.center);
		\draw (4.center) to (2.center);
		\draw (16.center) to (15.center);
		\draw [style=fill grey no line] (23.center)
			 to (22.center)
			 to (25.center)
			 to (24.center)
			 to cycle;
		\draw [style=fill grey no line] (27.center)
			 to (26.center)
			 to (29.center)
			 to (28.center)
			 to cycle;
		\draw [style=fill grey no line] (31.center)
			 to (30.center)
			 to (33.center)
			 to (32.center)
			 to cycle;
	\end{pgfonlayer}
\end{tikzpicture}$} .$$
In this case, the root of the tree corresponds to the top rectangle.

Via $\psi^\mathrm{Hawkins}$, a quilt $Q=(W,T)$ acts on GS cochains $(\te_1,\ldots,\te_n)$ by composing them vertically according to the tree $T$ matching up the rectangles horizontally via the word $W$. Again, the `open spaces' are filled in by instances of restrictions $f$. However, as the restrictions of presheaves are functorial, i.e. $f^{uv}= f^v f^u$ for two composable arrows $u$ and $v$, there is no need to involve their multiplications $m$ (nor vertical versions thereof). In particular, as there is a single rectangle at the bottom of $Q$, there is no need to compose with multiplications at the bottom.

We illustrate the action by an example.
\begin{vb}
For the quilt $Q$ and cochains $\underline{\te}=(\te_1,\ldots,\te_5)$ given in Example \ref{ex:action}, $\psi^{\mathrm{Hawkins}}(Q)(\underline{\te})=0$. If we replace cochain $\te_2$ by a cochain $\te_2'$ in $\CC^{3,1}(\A)$, we obtain
\begin{align*}
\psi^{\mathrm{Hawkins}}\left(\; \scalebox{0.7}{$\begin{tikzpicture}
	\begin{pgfonlayer}{nodelayer}
		\node [style=none] (0) at (1.5, -2) {};
		\node [style=none] (1) at (-1.5, -2) {};
		\node [style=none] (2) at (-1.5, -1) {};
		\node [style=none] (3) at (1.5, -1) {};
		\node [style=none] (4) at (-1.5, 2) {};
		\node [style=none] (5) at (-0.5, 2) {};
		\node [style=none] (6) at (-0.5, -1) {};
		\node [style=none] (7) at (0.5, 0) {};
		\node [style=none] (8) at (0.5, -1) {};
		\node [style=none] (9) at (-0.5, 0) {};
		\node [style=none] (10) at (1.5, 0) {};
		\node [style=none] (11) at (1.5, 1) {};
		\node [style=none] (12) at (0.5, 1) {};
		\node [style=none] (13) at (-0.5, 1) {};
		\node [style=none] (14) at (0.5, 2) {};
		\node [style=none] (15) at (-0.6, 0.125) {};
		\node [style=none] (16) at (-0.6, 1.875) {};
		\node [style=none] (17) at (0, -1.5) {$1$};
		\node [style=none] (18) at (-1, 0.5) {$2$};
		\node [style=none] (19) at (0, -0.5) {$3$};
		\node [style=none] (20) at (1, 0.5) {$4$};
		\node [style=none] (21) at (0, 1.5) {$5$};
		\node [style=none] (22) at (0.525, -0.975) {};
		\node [style=none] (23) at (0.525, -0.025) {};
		\node [style=none] (24) at (1.475, -0.025) {};
		\node [style=none] (25) at (1.475, -0.975) {};
		\node [style=none] (26) at (0.525, 1.025) {};
		\node [style=none] (27) at (0.525, 1.975) {};
		\node [style=none] (28) at (1.475, 1.975) {};
		\node [style=none] (29) at (1.475, 1.025) {};
		\node [style=none] (30) at (-0.475, 0.025) {};
		\node [style=none] (31) at (-0.475, 0.975) {};
		\node [style=none] (32) at (0.475, 0.975) {};
		\node [style=none] (33) at (0.475, 0.025) {};
	\end{pgfonlayer}
	\begin{pgfonlayer}{edgelayer}
		\draw (14.center) to (5.center);
		\draw (5.center) to (13.center);
		\draw (13.center) to (12.center);
		\draw (12.center) to (14.center);
		\draw (11.center) to (12.center);
		\draw (12.center) to (7.center);
		\draw (7.center) to (10.center);
		\draw (10.center) to (11.center);
		\draw (0.center) to (3.center);
		\draw (0.center) to (1.center);
		\draw (1.center) to (2.center);
		\draw (2.center) to (6.center);
		\draw (6.center) to (8.center);
		\draw (8.center) to (3.center);
		\draw (8.center) to (7.center);
		\draw (7.center) to (9.center);
		\draw (9.center) to (6.center);
		\draw (9.center) to (13.center);
		\draw (4.center) to (5.center);
		\draw (4.center) to (2.center);
		\draw (16.center) to (15.center);
		\draw [style=fill grey no line] (23.center)
			 to (22.center)
			 to (25.center)
			 to (24.center)
			 to cycle;
		\draw [style=fill grey no line] (27.center)
			 to (26.center)
			 to (29.center)
			 to (28.center)
			 to cycle;
		\draw [style=fill grey no line] (31.center)
			 to (30.center)
			 to (33.center)
			 to (32.center)
			 to cycle;
	\end{pgfonlayer}
\end{tikzpicture}$}  \;\right)(\te_1,\te_2',\te_3,\te_4,\te_5)^{\si}&= \quad \scalebox{0.9}{$\begin{tikzpicture}
	\begin{pgfonlayer}{nodelayer}
		\node [style=none] (341) at (1.175, 2.125) {};
		\node [style=none] (342) at (3.05, 2.125) {};
		\node [style=none] (343) at (1.175, 1.025) {};
		\node [style=none] (344) at (3.05, 1.025) {};
		\node [style=none] (337) at (-0.7, 4.425) {};
		\node [style=none] (338) at (1.175, 4.425) {};
		\node [style=none] (339) at (-0.7, 1.05) {};
		\node [style=none] (340) at (1.175, 1.05) {};
		\node [style=none] (333) at (3.1, 1) {};
		\node [style=none] (334) at (4.45, 1) {};
		\node [style=none] (335) at (3.1, -1.175) {};
		\node [style=none] (336) at (4.45, -1.175) {};
		\node [style=none] (203) at (-4.675, 4.325) {};
		\node [style=none] (204) at (-4.675, -1.125) {};
		\node [style=none] (209) at (-0.775, 4.425) {};
		\node [style=none] (210) at (-4.6, 4.425) {};
		\node [style=none] (211) at (-0.75, -1.2) {};
		\node [style=none] (212) at (-4.575, -1.2) {};
		\node [style=none] (215) at (-2.675, 1.475) {$\theta_2^{u_4u_5,u_6,u_7u_8}$};
		\node [style=none] (225) at (1.2, 1.1) {};
		\node [style=none] (227) at (-0.675, -1.1) {};
		\node [style=none] (228) at (-0.675, 0.925) {};
		\node [style=none] (229) at (2.975, 1.025) {};
		\node [style=none] (230) at (1.275, 1.025) {};
		\node [style=none] (231) at (1.1, 1.025) {};
		\node [style=none] (232) at (-0.6, 1.025) {};
		\node [style=none] (233) at (2.975, -1.175) {};
		\node [style=none] (234) at (-0.6, -1.175) {};
		\node [style=none] (235) at (1.375, -0.125) {$\theta_3^{u_4,u_5}$};
		\node [style=none] (237) at (1.175, 2.225) {};
		\node [style=none] (238) at (2.975, 4.425) {};
		\node [style=none] (239) at (1.275, 4.425) {};
		\node [style=none] (240) at (2.95, 2.125) {};
		\node [style=none] (241) at (1.25, 2.125) {};
		\node [style=none] (251) at (1.2, 1.975) {};
		\node [style=none] (252) at (1.2, 1.1) {};
		\node [style=none] (259) at (3.1, 1.025) {};
		\node [style=none] (260) at (3.1, -1.2) {};
		\node [style=none] (261) at (3.1, 4.425) {};
		\node [style=none] (262) at (4.45, 4.425) {};
		\node [style=none] (263) at (3.1, 2.15) {};
		\node [style=none] (264) at (4.45, 2.15) {};
		\node [style=none] (266) at (4.425, -0.175) {};
		\node [style=none] (267) at (4.425, -1.1) {};
		\node [style=none] (268) at (4.425, 0.925) {};
		\node [style=none] (269) at (4.425, 0) {};
		\node [style=none] (270) at (4.425, 2.025) {};
		\node [style=none] (271) at (4.425, 1.1) {};
		\node [style=none] (272) at (4.425, 3.15) {};
		\node [style=none] (273) at (4.425, 2.7) {};
		\node [style=none] (274) at (3.075, -0.175) {};
		\node [style=none] (275) at (3.075, -1.1) {};
		\node [style=none] (276) at (3.075, 0.925) {};
		\node [style=none] (277) at (3.075, 0) {};
		\node [style=none] (278) at (4.325, 3.25) {};
		\node [style=none] (279) at (3.15, 3.25) {};
		\node [style=none] (280) at (4.3, -0.075) {};
		\node [style=none] (281) at (3.125, -0.075) {};
		\node [style=none] (282) at (4.325, 1.025) {};
		\node [style=none] (283) at (3.15, 1.025) {};
		\node [style=none] (284) at (4.325, -1.175) {};
		\node [style=none] (285) at (3.15, -1.175) {};
		\node [style=none] (286) at (3.825, -0.625) {$f^{u_4}$};
		\node [style=none] (287) at (4.35, 2.125) {};
		\node [style=none] (288) at (3.175, 2.125) {};
		\node [style=none] (289) at (3.075, 2.025) {};
		\node [style=none] (290) at (3.075, 1.1) {};
		\node [style=none] (291) at (3.075, 3.15) {};
		\node [style=none] (292) at (3.075, 2.2) {};
		\node [style=none] (293) at (3.75, 2.675) {$f^{u_7}$};
		\node [style=none] (294) at (3.75, 1.575) {$\theta_5^{u_6}$};
		\node [style=none] (295) at (3.675, 0.475) {$f^{u_5}$};
		\node [style=none] (296) at (3.1, 4.425) {};
		\node [style=none] (297) at (4.45, 4.425) {};
		\node [style=none] (298) at (4.425, 4.325) {};
		\node [style=none] (299) at (4.425, 3.375) {};
		\node [style=none] (300) at (4.325, 4.425) {};
		\node [style=none] (301) at (3.15, 4.425) {};
		\node [style=none] (302) at (3.075, 4.325) {};
		\node [style=none] (303) at (3.075, 3.375) {};
		\node [style=none] (304) at (3.75, 3.85) {$f^{u_8}$};
		\node [style=none] (305) at (-0.675, 1.1) {};
		\node [style=none] (306) at (-0.7, 2.225) {};
		\node [style=none] (309) at (-0.7, 2.2) {};
		\node [style=none] (310) at (-0.7, 4.325) {};
		\node [style=none] (311) at (-0.675, 1.975) {};
		\node [style=none] (312) at (-0.675, 1.1) {};
		\node [style=none] (313) at (1.075, 4.425) {};
		\node [style=none] (314) at (-0.625, 4.425) {};
		\node [style=none] (315) at (1.1, 2.125) {};
		\node [style=none] (316) at (-0.6, 2.125) {};
		\node [style=none] (319) at (4.45, -3.55) {};
		\node [style=none] (320) at (4.45, -4.475) {};
		\node [style=none] (321) at (4.45, -2.45) {};
		\node [style=none] (322) at (4.45, -3.375) {};
		\node [style=none] (323) at (4.45, -1.35) {};
		\node [style=none] (324) at (4.45, -2.275) {};
		\node [style=none] (325) at (4.35, -4.575) {};
		\node [style=none] (326) at (-4.475, -4.575) {};
		\node [style=none] (327) at (-4.65, -1.35) {};
		\node [style=none] (328) at (-4.65, -4.45) {};
		\node [style=none] (329) at (-0.075, -2.95) {$\theta_1^{u_1,u_2,u_3}$};
		\node [style=none] (330) at (2.175, 1.575) {$f^{u_6}$};
		\node [style=none] (331) at (0.425, 1.575) {$f^{u_6}$};
		\node [style=none] (332) at (2.15, 3.225) {$\theta_4^{u_7,u_8}$};
		\node [style=none] (347) at (1.175, 4.325) {};
		\node [style=none] (350) at (1.175, 2.175) {};
		\node [style=none] (351) at (1.175, 2.2) {};
		\node [style=none] (352) at (1.175, 4.325) {};
		\node [style=none] (353) at (0.25, 3.25) {$f^{u_7 u_8}$};
	\end{pgfonlayer}
	\begin{pgfonlayer}{edgelayer}
		\draw [style=light yellow] (342.center)
			 to (341.center)
			 to (343.center)
			 to (344.center)
			 to cycle;
		\draw [style=light yellow] (337.center)
			 to (339.center)
			 to (340.center)
			 to (338.center)
			 to cycle;
		\draw [style=light yellow] (333.center)
			 to (335.center)
			 to (336.center)
			 to (334.center)
			 to cycle;
		\draw [style=arrow] (203.center) to (204.center);
		\draw [style=arrow] (209.center) to (210.center);
		\draw [style=arrow, in=360, out=180] (211.center) to (212.center);
		\draw [style=arrow] (228.center) to (227.center);
		\draw [style=arrow] (229.center) to (230.center);
		\draw [style=arrow] (231.center) to (232.center);
		\draw [style=arrow] (233.center) to (234.center);
		\draw [style=arrow] (238.center) to (239.center);
		\draw [style=arrow] (240.center) to (241.center);
		\draw [style=arrow] (251.center) to (252.center);
		\draw [style=light yellow] (260.center) to (259.center);
		\draw [style=light yellow] (261.center)
			 to (263.center)
			 to (264.center)
			 to (262.center)
			 to cycle;
		\draw [style=arrow] (266.center) to (267.center);
		\draw [style=arrow] (268.center) to (269.center);
		\draw [style=arrow] (270.center) to (271.center);
		\draw [style=arrow] (272.center) to (273.center);
		\draw [style=arrow] (274.center) to (275.center);
		\draw [style=arrow] (276.center) to (277.center);
		\draw [style=arrow] (278.center) to (279.center);
		\draw [style=arrow] (280.center) to (281.center);
		\draw [style=arrow] (282.center) to (283.center);
		\draw [style=arrow] (284.center) to (285.center);
		\draw [style=arrow] (287.center) to (288.center);
		\draw [style=arrow] (289.center) to (290.center);
		\draw [style=arrow] (291.center) to (292.center);
		\draw [style=light yellow] (297.center) to (296.center);
		\draw [style=arrow] (298.center) to (299.center);
		\draw [style=arrow] (300.center) to (301.center);
		\draw [style=arrow] (302.center) to (303.center);
		\draw [style=arrow] (310.center) to (309.center);
		\draw [style=arrow] (311.center) to (312.center);
		\draw [style=arrow] (313.center) to (314.center);
		\draw [style=arrow] (315.center) to (316.center);
		\draw [style=arrow] (319.center) to (320.center);
		\draw [style=arrow] (321.center) to (322.center);
		\draw [style=arrow] (323.center) to (324.center);
		\draw [style=arrow, in=360, out=180] (325.center) to (326.center);
		\draw [style=arrow] (327.center) to (328.center);
		\draw [style=arrow] (352.center) to (351.center);
	\end{pgfonlayer}
\end{tikzpicture}$}  \quad + \quad   \scalebox{0.9}{$\begin{tikzpicture}
	\begin{pgfonlayer}{nodelayer}
		\node [style=none] (353) at (-0.7, 2.1) {};
		\node [style=none] (354) at (1.3, 2.1) {};
		\node [style=none] (355) at (-0.7, 1.025) {};
		\node [style=none] (356) at (1.3, 1.025) {};
		\node [style=none] (341) at (1.3, 2.125) {};
		\node [style=none] (342) at (2.8, 2.125) {};
		\node [style=none] (343) at (1.3, 1.025) {};
		\node [style=none] (344) at (2.8, 1.025) {};
		\node [style=none] (337) at (1.3, 4.425) {};
		\node [style=none] (338) at (2.825, 4.425) {};
		\node [style=none] (339) at (1.3, 2.125) {};
		\node [style=none] (340) at (2.825, 2.125) {};
		\node [style=none] (333) at (2.85, 1) {};
		\node [style=none] (334) at (4.2, 1) {};
		\node [style=none] (335) at (2.85, -1.175) {};
		\node [style=none] (336) at (4.2, -1.175) {};
		\node [style=none] (203) at (-4.675, 4.325) {};
		\node [style=none] (204) at (-4.675, -1.125) {};
		\node [style=none] (209) at (-0.775, 4.425) {};
		\node [style=none] (210) at (-4.6, 4.425) {};
		\node [style=none] (211) at (-0.75, -1.2) {};
		\node [style=none] (212) at (-4.575, -1.2) {};
		\node [style=none] (215) at (-2.675, 1.475) {$\theta_2^{u_4u_5,u_6,u_7u_8}$};
		\node [style=none] (225) at (1.325, 1.1) {};
		\node [style=none] (227) at (-0.675, -1.1) {};
		\node [style=none] (228) at (-0.675, 0.925) {};
		\node [style=none] (229) at (2.725, 1.025) {};
		\node [style=none] (230) at (1.4, 1.025) {};
		\node [style=none] (231) at (1.225, 1.025) {};
		\node [style=none] (232) at (-0.6, 1.025) {};
		\node [style=none] (233) at (2.725, -1.175) {};
		\node [style=none] (234) at (-0.6, -1.175) {};
		\node [style=none] (235) at (1.375, -0.125) {$\theta_3^{u_4,u_5}$};
		\node [style=none] (237) at (1.3, 2.225) {};
		\node [style=none] (238) at (2.725, 4.425) {};
		\node [style=none] (239) at (1.4, 4.425) {};
		\node [style=none] (240) at (2.7, 2.125) {};
		\node [style=none] (241) at (1.375, 2.125) {};
		\node [style=none] (251) at (1.325, 1.975) {};
		\node [style=none] (252) at (1.325, 1.1) {};
		\node [style=none] (259) at (2.85, 1.025) {};
		\node [style=none] (260) at (2.85, -1.2) {};
		\node [style=none] (261) at (2.85, 4.425) {};
		\node [style=none] (262) at (4.2, 4.425) {};
		\node [style=none] (263) at (2.85, 2.15) {};
		\node [style=none] (264) at (4.2, 2.15) {};
		\node [style=none] (266) at (4.175, -0.175) {};
		\node [style=none] (267) at (4.175, -1.1) {};
		\node [style=none] (268) at (4.175, 0.925) {};
		\node [style=none] (269) at (4.175, 0) {};
		\node [style=none] (270) at (4.175, 2.025) {};
		\node [style=none] (271) at (4.175, 1.1) {};
		\node [style=none] (272) at (4.175, 3.15) {};
		\node [style=none] (273) at (4.175, 2.2) {};
		\node [style=none] (274) at (2.825, -0.175) {};
		\node [style=none] (275) at (2.825, -1.1) {};
		\node [style=none] (276) at (2.825, 0.925) {};
		\node [style=none] (277) at (2.825, 0) {};
		\node [style=none] (278) at (4.075, 3.25) {};
		\node [style=none] (279) at (2.9, 3.25) {};
		\node [style=none] (280) at (4.05, -0.075) {};
		\node [style=none] (281) at (2.875, -0.075) {};
		\node [style=none] (282) at (4.075, 1.025) {};
		\node [style=none] (283) at (2.9, 1.025) {};
		\node [style=none] (284) at (4.075, -1.175) {};
		\node [style=none] (285) at (2.9, -1.175) {};
		\node [style=none] (286) at (3.575, -0.625) {$f^{u_4}$};
		\node [style=none] (287) at (4.1, 2.125) {};
		\node [style=none] (288) at (2.925, 2.125) {};
		\node [style=none] (289) at (2.825, 2.025) {};
		\node [style=none] (290) at (2.825, 1.1) {};
		\node [style=none] (291) at (2.825, 3.15) {};
		\node [style=none] (292) at (2.825, 2.2) {};
		\node [style=none] (293) at (3.5, 2.675) {$f^{u_7}$};
		\node [style=none] (294) at (3.5, 1.575) {$\theta_5^{u_6}$};
		\node [style=none] (295) at (3.425, 0.475) {$f^{u_5}$};
		\node [style=none] (296) at (2.85, 4.425) {};
		\node [style=none] (297) at (4.2, 4.425) {};
		\node [style=none] (298) at (4.175, 4.325) {};
		\node [style=none] (299) at (4.175, 3.375) {};
		\node [style=none] (300) at (4.075, 4.425) {};
		\node [style=none] (301) at (2.9, 4.425) {};
		\node [style=none] (302) at (2.825, 4.325) {};
		\node [style=none] (303) at (2.825, 3.375) {};
		\node [style=none] (304) at (3.5, 3.85) {$f^{u_8}$};
		\node [style=none] (305) at (-0.675, 1.1) {};
		\node [style=none] (306) at (-0.7, 2.225) {};
		\node [style=none] (309) at (-0.7, 2.2) {};
		\node [style=none] (310) at (-0.7, 4.325) {};
		\node [style=none] (311) at (-0.675, 1.975) {};
		\node [style=none] (312) at (-0.675, 1.1) {};
		\node [style=none] (313) at (1.2, 4.425) {};
		\node [style=none] (314) at (-0.625, 4.425) {};
		\node [style=none] (315) at (1.225, 2.125) {};
		\node [style=none] (316) at (-0.6, 2.125) {};
		\node [style=none] (319) at (4.2, -3.55) {};
		\node [style=none] (320) at (4.2, -4.475) {};
		\node [style=none] (321) at (4.2, -2.45) {};
		\node [style=none] (322) at (4.2, -3.375) {};
		\node [style=none] (323) at (4.2, -1.35) {};
		\node [style=none] (324) at (4.2, -2.275) {};
		\node [style=none] (325) at (4.1, -4.575) {};
		\node [style=none] (326) at (-4.475, -4.575) {};
		\node [style=none] (327) at (-4.65, -1.35) {};
		\node [style=none] (328) at (-4.65, -4.45) {};
		\node [style=none] (329) at (-0.075, -2.95) {$\theta_1^{u_1,u_2,u_3}$};
		\node [style=none] (330) at (2.175, 1.575) {$f^{u_6}$};
		\node [style=none] (331) at (0.425, 1.575) {$f^{u_6}$};
		\node [style=none] (332) at (0.3, 3.2) {$\theta_4^{u_7,u_8}$};
		\node [style=none] (345) at (2.625, 3.25) {};
		\node [style=none] (346) at (1.45, 3.25) {};
		\node [style=none] (347) at (1.3, 4.325) {};
		\node [style=none] (348) at (1.3, 3.375) {};
		\node [style=none] (349) at (1.3, 3.125) {};
		\node [style=none] (350) at (1.3, 2.175) {};
		\node [style=none] (351) at (2.075, 2.675) {$f^{u_7}$};
		\node [style=none] (352) at (2.075, 3.85) {$f^{u_8}$};
	\end{pgfonlayer}
	\begin{pgfonlayer}{edgelayer}
		\draw [style=light yellow] (353.center)
			 to (355.center)
			 to (356.center)
			 to (354.center)
			 to cycle;
		\draw [style=light yellow] (342.center)
			 to (341.center)
			 to (343.center)
			 to (344.center)
			 to cycle;
		\draw [style=light yellow] (337.center)
			 to (339.center)
			 to (340.center)
			 to (338.center)
			 to cycle;
		\draw [style=light yellow] (333.center)
			 to (335.center)
			 to (336.center)
			 to (334.center)
			 to cycle;
		\draw [style=arrow] (203.center) to (204.center);
		\draw [style=arrow] (209.center) to (210.center);
		\draw [style=arrow, in=360, out=180] (211.center) to (212.center);
		\draw [style=arrow] (228.center) to (227.center);
		\draw [style=arrow] (229.center) to (230.center);
		\draw [style=arrow] (231.center) to (232.center);
		\draw [style=arrow] (233.center) to (234.center);
		\draw [style=arrow] (238.center) to (239.center);
		\draw [style=arrow] (240.center) to (241.center);
		\draw [style=arrow] (251.center) to (252.center);
		\draw [style=light yellow] (260.center) to (259.center);
		\draw [style=light yellow] (261.center)
			 to (263.center)
			 to (264.center)
			 to (262.center)
			 to cycle;
		\draw [style=arrow] (266.center) to (267.center);
		\draw [style=arrow] (268.center) to (269.center);
		\draw [style=arrow] (270.center) to (271.center);
		\draw [style=arrow] (272.center) to (273.center);
		\draw [style=arrow] (274.center) to (275.center);
		\draw [style=arrow] (276.center) to (277.center);
		\draw [style=arrow] (278.center) to (279.center);
		\draw [style=arrow] (280.center) to (281.center);
		\draw [style=arrow] (282.center) to (283.center);
		\draw [style=arrow] (284.center) to (285.center);
		\draw [style=arrow] (287.center) to (288.center);
		\draw [style=arrow] (289.center) to (290.center);
		\draw [style=arrow] (291.center) to (292.center);
		\draw [style=light yellow] (297.center) to (296.center);
		\draw [style=arrow] (298.center) to (299.center);
		\draw [style=arrow] (300.center) to (301.center);
		\draw [style=arrow] (302.center) to (303.center);
		\draw [style=arrow] (310.center) to (309.center);
		\draw [style=arrow] (311.center) to (312.center);
		\draw [style=arrow] (313.center) to (314.center);
		\draw [style=arrow] (315.center) to (316.center);
		\draw [style=arrow] (319.center) to (320.center);
		\draw [style=arrow] (321.center) to (322.center);
		\draw [style=arrow] (323.center) to (324.center);
		\draw [style=arrow, in=360, out=180] (325.center) to (326.center);
		\draw [style=arrow] (327.center) to (328.center);
		\draw [style=arrow] (345.center) to (346.center);
		\draw [style=arrow] (347.center) to (348.center);
		\draw [style=arrow] (349.center) to (350.center);
	\end{pgfonlayer}
\end{tikzpicture}$} 
\end{align*} 
\end{vb}
\begin{opm}
Observe that we have two distinctly different actions of $\Quilt$ on the GS complex for a presheaf: one that employs both the data of $m$ and $f$, and one that employs solely the datum $f$.
\end{opm}

\subsubsection{Another morphism $\TwQuilt \longrightarrow \End(s^{-1}\CGS(\A))$} 
\label{parparparotheractiontwistquilt}

In \cite{hawkins}, Hawkins extends the action of $\Quilt$ on $s^{-1}\CGS(\A)$ by twisting with the multiplications $m$ and employing the operad $\mQuilt$, which we can rephrase as follows: $\mQuilt$ is the completed coproduct of $\Quilt$ and a formal element $m$ of arity $0$ and degree $1$ imposing the following relations:
\begin{enumerate}
\item $\partial(m) = 0$,
\item $L_2(m,m) = 0$,
\item \label{mquilt3} $Q \circ_i m = 0$ if $i$ has either more than two vertical children, has at least one horizontal child or is a parent horizontally. 
\item $Q \circ_i m  = Q' \circ_i m$ if $Q=(T,W)$ and $Q'=(T,W')$ and $W$ and $W'$ differ solely in the position of $i$. 
\end{enumerate}
The action $\psi^{\mathrm{Hawkins}}$ extends to $\mQuilt$ by sending the formal element $m$ to the multiplication $m\in \CC^{0,2}(\A)$, obtaining a morphism of dg operads
$$\psi^{\mathrm{Hawkins}}_c:\mQuilt \longrightarrow \End(s^{-1}\CGS(\A))$$
where, this time, the GS complex is endowed with the full differential $d_0 + d_1$. Further, Hawkins obtains a new morphism $\Linf \longrightarrow \mQuilt:l_n \longmapsto L^m_n$ via twisting with $m$: define for $n\geq 1$ 
$$ L_n^m := L_n + (-1)^{n+1}L_{n+1}(m,-,\ldots,-).$$
The new differential is given by
$$\partial^m = \partial + \partial_{L^m_1} .$$
\begin{lemma}
We have a morphism of dg operads $\TwQuilt \longrightarrow \mQuilt$ that is the identity on quilts and sends $\alpha$ to $m$. 
\end{lemma}
The following constitutes another solution to the Deligne conjecture for presheaves of categories.
\begin{theorem}
The operad $\TwQuilt$ acts on the desuspended GS complex, i.e. we have a morphism of dg operads
$$\TwQuilt \longrightarrow \End(s^{-1}\CGS(\A)).$$
\end{theorem}

\def\cprime{$'$}
\providecommand{\bysame}{\leavevmode\hbox to3em{\hrulefill}\thinspace}
\bibliography{Bibfile}
\bibliographystyle{amsalpha}

\end{document}